\documentclass[11pt]{amsart}
\usepackage{geometry}                
\usepackage{graphicx}
\usepackage{amssymb}
\usepackage{epstopdf}
\usepackage{tikz, pgfplots}
\usepackage{tikz-cd} 
\usepackage{color}
\usepackage{xcolor}
\usepackage{hyperref}
\usepackage{enumitem}
\usepackage{bm}
\usepackage{mathrsfs} 
\usepackage{booktabs} 

\usepackage{soul}
\usepackage{cancel}
\usepackage[normalem]{ulem} 

\DeclareMathAlphabet{\mathcalligra}{T1}{calligra}{m}{n}

\DeclareMathOperator{\CA}{\mathrm{CA}}

\newcommand{\tfr}{\mathfrak{t}}
\newcommand{\gfr}{\mathfrak{g}}
\newcommand{\hfr}{\mathfrak{h}}
\newcommand{\C}{\mathbb{C}}
\newcommand{\R}{\mathbb{R}}
\newcommand{\Z}{\mathbb{Z}}
\newcommand{\Q}{\mathbb{Q}}

\newcommand{\Cc}{\mathbf{C}}
\newcommand{\Bbo}{\mathbf{B}^\circ}
\newcommand{\Bb}{\mathbf{B}}
\newcommand{\bsigma}{\bm{\sigma}}
\newcommand{\bLambda}{\bm{\Lambda}}

\newcommand{\X}{\mathcal{X}}

\newcommand{\cu}{\Sigma}

\newcommand{\Ub}{\bm{U}}

\newcommand{\bq}{\bm{q}}

\newcommand{\bc}{\bm{\chi}}
\newcommand{\mh}{\mathcal{M}}
\newcommand{\tcu}{\widetilde{\cu}}

\renewcommand{\to}{\longrightarrow}
\newcommand{\id}{\mathrm{id}}

\definecolor{cherry}{rgb}{0.9,.1,.2}
\definecolor{purple}{rgb}{0.6,0,0.6}
\definecolor{bl}{rgb}{0.13,0.67,0.9} 
\definecolor{tur}{rgb}{0.1,0.9,0.9}
\definecolor{gr}{rgb}{0.2,0.7,0.2}
\definecolor{pink}{rgb}{1,0,0.8}
\definecolor{Darkblue}{rgb}{0.2, 0.2, 0.6}

\newcounter{save}
\newcounter{rmk}

\theoremstyle{plain}
\newtheorem{thm}{Theorem}[subsection]
\newtheorem{prop}[thm]{Proposition}
\newtheorem{lem}[thm]{Lemma}
\newtheorem{cor}[thm]{Corollary}

\newtheorem{introthmcmd}{Theorem}
\newenvironment{introthm}[1]
{\renewcommand\theintrothmcmd{#1}\introthmcmd}
{\endintrothmcmd}

\theoremstyle{definition}
\newtheorem{dfn}[thm]{Definition}
\newtheorem{ex}[thm]{Example}
\newtheorem*{question*}{Question}

\theoremstyle{remark}
\newtheorem{rem}[thm]{Remark}

\numberwithin{equation}{section}

\author{Florian Beck}
\author{Ron Donagi}
\author{Katrin Wendland}
\title{Folding of Hitchin systems and crepant resolutions}

\begin{document}
\maketitle
\vspace{-0.6cm}
\begin{abstract}
	Folding of ADE-Dynkin diagrams according to graph automorphisms yields irreducible Dynkin diagrams of $\mathrm{ABCDEFG}$-types. 
	This folding procedure  allows to trace back the properties of the corresponding simple 
	Lie algebras or groups to those of $\mathrm{ADE}$-type. 
	In this article, we implement the techniques of folding by graph automorphisms
	for Hitchin integrable systems.
	We show that the fixed point loci of these automorphisms
	are isomorphic as algebraic integrable systems to the 
	Hitchin systems of the folded groups away from singular fibers. 
	The latter Hitchin systems are isomorphic to the intermediate Jacobian
	fibrations of Calabi--Yau orbifold stacks constructed by the first author. 
	We construct simultaneous crepant resolutions of the associated singular quasi-projective Calabi--Yau threefolds and compare the resulting intermediate Jacobian fibrations to the corresponding Hitchin systems. 
\end{abstract}
\vspace{-.8cm}
\tableofcontents

\section{Introduction}\label{intro}
Any non-simply-laced Dynkin diagram $\Delta$ is obtained from a 
simply-laced
one $\Delta_h$ (i.e.~of type $\mathrm{ADE}$\footnote{Some authors, 
for example Slodowy \cite{Slo}, call the $\mathrm{ADE}$- or simply-laced 
Dynkin diagrams `homogeneous'. This explains the subscript $h$.}) by folding.
Folding is the process of identifying nodes of $\Delta_h$ according to a cyclic 
subgroup $\Cc\subset \mathrm{Aut}(\Delta_h)$ of the graph automorphism of $\Delta_h$, for example:
\begin{figure}[h]
	\scalebox{0.85}{
		\begin{tikzpicture}
		\coordinate (A) at (0,0) {}; 
		\coordinate (B) at (1,0) {}; 
		\coordinate (C) at (2,0) {}; 
		\coordinate (D) at (3,0) {}; 
		\coordinate (E) at (4,0) {}; 
		\coordinate (A5) at (-1.5,.25) {}; 
		\coordinate (Cc) at (-1.5,-.25) {};
		
		\coordinate (F1) at (4.75,0) {}; 
		\coordinate (G1) at (6.25,0) {};
		\coordinate (H1) at (7,0) {};
		\coordinate (I1) at (8,0) {}; 
		\coordinate (J1) at (9,0) {};
		\coordinate (Bhalf) at (1.5, -0.9) {}; 
		\coordinate (Hhalf) at (1.5,-1.45) {}; 
		
		\coordinate (C3) at (11,0) {}; 
		
		\node [fill=black, circle, inner sep=0pt, minimum size=5pt] at (A) {}; 
		\node [fill=black, circle, inner sep=0pt, minimum size=5pt] at (B) {};
		\node [fill=black, circle, inner sep=0pt, minimum size=5pt] at (C) {};
		\node [fill=black, circle, inner sep=0pt, minimum size=5pt] at (D) {};
		\node [fill=black, circle, inner sep=0pt, minimum size=5pt] at (E) {}; 
		\node [fill=black, circle, inner sep=0pt, minimum size=5pt] at (H1) {}; 
		\node [fill=black, circle, inner sep=0pt, minimum size=5pt] at (I1) {}; 
		\node [fill=black, circle, inner sep=0pt, minimum size=5pt] at (J1) {}; 
		\node at (A5) {$\Delta_h=\mathrm{A}_5$}; 
		\node at (C3) {$\Delta=\Delta_{h,\Cc}=\mathrm{C}_3$};
		\node at (Cc) {$\Cc=\Z/2\Z$}; 
		
		\draw (A) -- (B);
		\draw (B) -- (C); 
		\draw (C) -- (D); 
		\draw (C) -- (E); 
		\draw [<->, shorten <= 0.15cm, shorten >=0.15cm, in=90, out=90] (A) to (E);
		\draw [<->, shorten <= 0.15cm, shorten >=0.15cm, in=90, out=90] (B) to (D);
		\draw (H1) -- (I1);
		\draw[<->]  (F1) -- node[midway] {} (G1); 
		
		\draw
		(8,0.04) -- (9,.04)
		(8,-0.04) -- (9,-.04)
		(8.4,0) --++  (50:.2) 
		(8.4,0)  --++  (-50:.2);
		
		\end{tikzpicture}}
	\caption{Folding of $\Delta_h=\mathrm{A}_5$ to $\Delta=\Delta_{h,\Cc}=\mathrm{C}_3$.}
	\label{Figure1}
\end{figure}

In Lie theory, folding effectively reduces the study of simple complex Lie algebras 
$\gfr=\gfr(\Delta)$ with folded Dynkin diagram $\Delta=\Delta_{h,\Cc}$ to simple complex Lie algebras $\gfr_h=\gfr(\Delta_h)$ with an $\mathrm{ADE}$-Dynkin diagram $\Delta_h$ and a lift of $\Cc\subset \mathrm{Aut}(\Delta_h)$ to outer automorphisms of $\gfr_h$. 
The same applies to simple complex Lie groups. 

In singularity theory, Slodowy \cite{Slo} used folding to \emph{define} a 
$\Delta$-singularity of a surface for any irreducible 
Dynkin diagram $\Delta$, thereby generalizing $\mathrm{ADE}$-surface singularities. 
The idea is as before: if $\Delta=\Delta_{h,\Cc}$, then a $\Delta$-singularity 
is a $\Delta_h$-singularity $Y$ together with an appropriate lift of $\Cc$ to $\mathrm{Aut}(Y)$.
The dual graph of the exceptional divisor of the minimal resolution of $Y$ coincides with $\Delta_h$. 
Then the lift of the $\Cc$-action to $\mathrm{Aut}(Y)$ induces a $\Cc$-action on $\Delta_h$ which is required to agree with the original $\Cc$-action on $\Delta_h$ by graph automorphisms.
\\

Since the $\mathrm{ADE}$-classification through simply-laced Dynkin
diagrams is ubiquitous in mathematics, it is natural to expect applications of
folding in other situations as well. Indeed, the $\mathrm{ADE}$-surface singularities
have been directly linked to $\mathrm{ADE}$-Hitchin integrable systems
by the second author with Diaconescu and Pantev in \cite{DDP}. Building
on results of Szendr\H oi's \cite{Sz1}, to each $\mathrm{ADE}$-surface singularity one associates a family of non-compact Calabi--Yau threefolds obtained from the semi-universal $\C^\ast$-deformation of the singularity, whose Griffiths' intermediate Jacobians are compact and together form an algebraic integrable system. The latter is called a \emph{Calabi--Yau integrable system},
and it generalizes integrable systems construced  in \cite{DM1} by the second author with Markman
from compact Calabi--Yau threefolds. According to \cite{DDP}, 
the Calabi--Yau integrable system obtained from an $\mathrm{ADE}$-surface
singularity is isomorphic to an $\mathrm{ADE}$-Hitchin integrable system.

Motivated by these results, the first aim of this article is to develop folding of Hitchin systems.
More specifically, let $H$ denote a reductive complex Lie group and $\mh(\cu,H)$ 
the neutral component of the moduli space of semistable $H$-Higgs bundles 
over a compact Riemann surface $\cu$ of genus 
\fontdimen16\textfont2=4pt
$g_\cu\geq 2$. 
\fontdimen16\textfont2=2.5pt
The smooth locus $\mh(\cu,H)^{sm}\subset \mh(\cu,H)$ possesses a holomorphic symplectic structure $\omega_H$ and carries the structure of an algebraic integrable system
\begin{equation*}
\begin{tikzcd}
\bm{\chi}\colon (\mh(\cu,H)^{sm},\omega_H)\to \Bb(\cu, H),
\end{tikzcd}
\end{equation*}
the $H$-Hitchin system. 
The Hitchin map $\bm{\chi}$ is a global analogue of the adjoint quotient 
\begin{equation}\label{eq:introadjoint}
	\chi\colon \mathfrak{h}\to \mathfrak{t}/W
\end{equation}
where $\mathfrak{h}$ is the Lie algebra of $H$ and $\tfr\subset \mathfrak{h}$ is a Cartan subalgebra with Weyl group $W$. 
Then the Hitchin base $\Bb(\cu,H)$ is given by $H^0(\cu,K_\cu\times_{\C^*}\tfr/W)$, where the canonical bundle $K_\cu$ of $\cu$ is considered as a $\C^*$-bundle and $\tfr/W$ is equipped with its natural $\C^*$-action. 

By folding of Hitchin systems, we mean the following: 
let $\Delta=\Delta_{h,\Cc}$ be an irreducible folded Dynkin diagram
and let $G$ and $G_h$ be the simple complex Lie groups 
of adjoint type with Dynkin diagrams $\Delta$ and $\Delta_h$ respectively. 
Then $\Cc$ lifts to a group of outer automorphisms of $G_h$ 
(see e.g. \cite[\S 10.3]{Springer}) which in turn induces a $\Cc$-action on $\mh(\cu,G_h)$. 
Our first main result is 
\begin{introthm}{A}[Theorem \ref{thm:hitchinfibers}]\label{thm:intro1}
	Over a Zariski-open and dense subset $\Bb^\circ\subset \Bb(\cu,G)$, the restriction $\mh(\cu, G)_{|\Bbo}$ of $\mh(\cu, G)$ is isomorphic to $\mh(\cu,G_h)_{|\Bbo}^\Cc$ as an algebraic integrable system.
\end{introthm}
The part of the Theorem that isomorphically identifies 
$\Bb^{\circ}$ with an open subset of $\Bb(\cu, G_h)^\Cc$ 
is more or less covered by the known commutative diagram
\begin{equation}\label{eq:introadjointquot}
\begin{tikzcd}
\gfr\ar[r,"\cong"] \ar[d,"\chi"] & \gfr_h^\Cc \ar[d, "\chi_h"]  \\
\tfr/W \ar[r, "\cong"] & (\tfr_h/W_h)^\Cc
\end{tikzcd}
\end{equation}
where all $\Cc$-actions are induced from the $\Cc$-action on $\Delta_h$ (see for example \cite[\S 8.8]{Slo} and our Section \ref{ss:foldingLAs}).
Theorem \ref{thm:intro1} is then a global version of \eqref{eq:introadjoint}.
The key new ingredient is the relation between cameral covers of $G$ and $G_h$, 
which we work out in Proposition \ref{p:foldingcameral} in Section \ref{ss:foldingcameral}.

In Section \ref{s:crepantresol}, we turn to the above-mentioned Calabi--Yau integrable systems associated to Hitchin systems \cite{DDP,Beck,Beck2}.
More specifically, the first author in \cite{Beck2} showed that the outer automorphism group $\Cc$ acts on the family $\mathcal{X}\to \Bb(\cu,G)$ of quasi-projective Gorenstein threefolds which is isomorphic to the restriction of the family $\X_h\longrightarrow \Bb(\cu, G_h)$ from \cite{DDP} to $\Bb(\cu, G_h)^\Cc\cong\Bb(\cu,G)$.
Moreover, the canonical class of these threefolds is $\Cc$-trivializable, thereby inducing the family $[\mathcal{X}/\Cc]\to \Bb(\cu,G)$ of Calabi--Yau orbifold stacks.
In \cite{Beck2}, the first author further showed that over an open and Zariski-dense subset $\Bbo\subset \Bb(\cu,G)$
with $\X^\circ:=\X_{|\Bb^\circ}$ the associated intermediate Jacobian fibration
	\begin{equation}\label{J2C}
	\begin{gathered}
	J^2_{\Cc}(\X^\circ/\Bb^\circ )\to \Bbo, 
	\\
	J^2_{\Cc}(X_b)=H^3(X_b,\C)^\Cc/(F^2H^3(X_b,\C)^\Cc+H^3(X_b,\Z)^\Cc),\; b\in \Bbo,
	\end{gathered}
	\end{equation}
is a Calabi--Yau integrable system.
The group $H^3(X_b,\Z)^\Cc$ of $\Cc$-invariants is isomorphic to 
the third integral equivariant cohomology group $H^3_{\Cc}(X_b,\Z)$ \cite[Proposition 4]{Beck2},
which is the third singular cohomology group $H^3([X_b/\Cc],\Z)$ of the orbifold stack 
$[X_b/\Cc]$, see \cite{Behrend-stacks}.
The main result of \cite{Beck} is the construction of an isomorphism 
\begin{equation}\label{eq:introiso}
J^2_{\Cc}(\X^\circ/\Bb^\circ)\cong \mh(\cu,G)_{|\Bbo}
\end{equation}
of algebraic integrable systems over $\Bbo$, see \cite[Theorem 5.2.1]{Beck}.
This generalizes \cite[Theorem 3]{DDP} where $\Cc=\{\id\}$, i.e. $G$ has an $\mathrm{ADE}$-Dynkin diagram.

Our second main result concerns crepant resolutions of the quotient varieties 
$X_b/\Cc$ associated to $[X_b/\Cc]$, $b\in \Bbo$. One might expect 
that, if they exist, their intermediate Jacobian fibrations 
are isomorphic to $J^2_{\Cc}(X_b)$  and hence to Hitchin fibers. 
However, we show: 
\begin{introthm}{B}[Proposition \ref{prop:simultaneous} and
Corollary \ref{cor:JZ}]\label{thm:intro2}
Let $\Cc \subset \mathrm{Aut}(\Delta_h)$ be a cyclic subgroup generated by an automorphism $a$ of order $|a|>0$ such that $\Delta=\Delta_{h,\Cc}$.
Then the family $\X^\circ/\Cc\to \Bbo$ admits a \emph{simultaneous crepant resolution} $\mathcal{Z}\to \Bbo$. 
Moreover, there is an isogeny
\begin{equation*}
J^2(\mathcal{Z}/\Bb^\circ) 
\simeq J^2_\Cc(\X^\circ/\Bb^\circ)\times J(\mathcal{X}_{|\Bbo}^\Cc)^{|a|-1}
\end{equation*}
over $\Bbo$.
Here $J(\mathcal{X}_{|\Bbo}^\Cc)\to \Bbo$ is the family of Jacobians of the fixed point loci $X_b^\Cc\subset X_b$, $b\in \Bbo$. 
\end{introthm}
We point out that the fixed point loci $X_b^\Cc$, $b\in \Bbo$, are branched coverings of $\cu$ and therefore have genus strictly larger than 
\fontdimen16\textfont2=4pt
$g_\cu$. 
\fontdimen16\textfont2=2.5pt
Because of \eqref{eq:introiso}, $J^2(\mathcal{Z}/\Bb^\circ)$ contains the $G$-Hitchin system $\mh(\cu,G)_{|\Bbo}$ over $\Bbo$ (up to isogeny), but this inclusion is proper.
Therefore $J^2(\mathcal{Z}/\Bb^\circ)$ cannot be an algebraic integrable system over $\Bbo$ for dimensional reasons.
It would be interesting to extend $J^2(\mathcal{Z}/\Bb^\circ)\to \Bbo$ to an algebraic integrable system over a larger base, a problem which Krichever solved for a special case  in a different context \cite{Kr05}.
\\

The proofs of Theorems \ref{thm:intro1} and \ref{thm:intro2} rely on a precise understanding of folding in Lie theory and in singularity theory. 
Since we were not able to locate a coherent account thereof in the literature, we provide one for the convenience of the reader. 
As they may have noticed, there are in fact two ways of 
folding an $\mathrm{ADE}$-Dynkin diagram $\Delta_h$ according to a cyclic subgroup $\Cc\subset \mathrm{Aut}(\Delta_h)$:
one corresponds to taking $\Cc$-coinvariants in the root system 
corresponding to $\Delta_h$, resulting in $\Delta=\Delta_{h,\Cc}$ 
(as depicted in Figure \ref{Figure1}), the other one to taking $\Cc$-invariants, resulting in 
the dual Dynkin diagram $\Delta^{\!\!\vee}= \Delta_h^\Cc$. 
We explain the interplay between these two procedures in Section \ref{s:folding}.
In particular, Figure \ref{Figure1} is extended to Figure \ref{Figure2} below.

\subsection{Relation to other works}\label{otherwork}
Folding in Lie theory is well-known (see for example \cite[\S 10.3]{Springer}). 
However, we could not locate a concise account in the literature which considers the adjoint quotients \eqref{eq:introadjointquot} as well.

Our construction of the folding of Hitchin systems (Theorem \ref{thm:intro1}) 
is closely related to the ideas of \cite{GPR}, where the action of outer 
automorphisms of $G$ on the entire total space $\mh(\cu,G)$ is considered. 
However, our result makes a statement about the fixed point locus 
relative to the Hitchin base $\Bb(\cu,G)$ which is not considered in \cite{GPR}, cf. Remark \ref{rem:GPR1}. 
In particular, the folding of cameral curves (Section \ref{ss:foldingcameral}) has not been considered before.

The families $\mathcal{X}\to \Bb(\cu,G)$ have been constructed in \cite{Beck2}. 
As pointed out before, \cite{Beck2} works directly with the global orbifold stacks $[X_b/\Cc]$, $b\in \Bbo$, instead of the crepant resolutions of the quotient varieties $X_b/\Cc$.

\begin{figure}[h]
	\scalebox{0.85}{
		\begin{tikzpicture}
		\coordinate (A) at (0,0) {}; 
		\coordinate (B) at (1,0) {}; 
		\coordinate (C) at (2,0) {}; 
		\coordinate (D) at (3,0) {}; 
		\coordinate (E) at (4,0) {}; 
		\coordinate (A5) at (-1.5,.25) {}; 
		\coordinate (Cc) at (-1.5,-.25) {};
		
		\coordinate (F1) at (4.75,.25) {}; 
		\coordinate (F2) at (4.75,-.25) {};
		\coordinate (G1) at (7.25,0.75) {};
		\coordinate (G2) at (7.25,-0.75){};
		\coordinate (H1) at (8,.75) {};
		\coordinate (H2) at (8,-.75) {};
		\coordinate (I1) at (9,.75) {}; 
		\coordinate (I2) at (9,-.75) {}; 
		\coordinate (J1) at (10,.75) {};
		\coordinate (J2) at (10,-.75) {};
		\coordinate (Bhalf) at (1.5, -0.9) {}; 
		\coordinate (Hhalf) at (1.5,-1.45) {}; 
		
		\coordinate (C3) at (12,.75) {}; 
		\coordinate (B3) at (12,-.75) {}; 
		
		\node [fill=black, circle, inner sep=0pt, minimum size=5pt] at (A) {}; 
		\node [fill=black, circle, inner sep=0pt, minimum size=5pt] at (B) {};
		\node [fill=black, circle, inner sep=0pt, minimum size=5pt] at (C) {};
		\node [fill=black, circle, inner sep=0pt, minimum size=5pt] at (D) {};
		\node [fill=black, circle, inner sep=0pt, minimum size=5pt] at (E) {}; 
		\node [fill=black, circle, inner sep=0pt, minimum size=5pt] at (H1) {}; 
		\node [fill=black, circle, inner sep=0pt, minimum size=5pt] at (I1) {}; 
		\node [fill=black, circle, inner sep=0pt, minimum size=5pt] at (J1) {}; 
		\node [fill=black, circle, inner sep=0pt, minimum size=5pt] at (H2) {}; 
		\node [fill=black, circle, inner sep=0pt, minimum size=5pt] at (I2) {}; 
		\node [fill=black, circle, inner sep=0pt, minimum size=5pt] at (J2) {}; 
		\node at (A5) {$\Delta_h=\mathrm{A}_5$}; 
		\node at (B3) {$\Delta^{\!\!\vee}=\Delta_{h}^\Cc=\mathrm{B}_3$}; 
		\node at (C3) {$\Delta=\Delta_{h,\Cc}=\mathrm{C}_3$};
		\node at (Cc) {$\Cc=\Z/2\Z$}; 
		
		\draw (A) -- (B);
		\draw (B) -- (C); 
		\draw (C) -- (D); 
		\draw (C) -- (E); 
		\draw [<->, shorten <= 0.15cm, shorten >=0.15cm] (A) to [in=90, out=90]  (E);
		\draw [<->, shorten <= 0.15cm, shorten >=0.15cm] (B) to [in=90, out=90] (D);
		\draw (H1) -- (I1);
		\draw (H2) -- (I2);
		\draw[<->]  (F1) -- node[midway, sloped, anchor=south] {\footnotesize{$\Cc$-coinvariants}} (G1); 
		\draw[<->] (F2) -- node[midway, sloped, anchor=north] {\footnotesize{$\Cc$-invariants}}  (G2); 
		
		\draw(9,0.71) -- (10,0.71);
		\draw(9,0.79) -- (10,0.79);
		\draw(9,-0.79) -- (10,-0.79);
		\draw(9,-.71) -- (10,-.71);

		\draw
		(9.6,-.75) --++  (130:.2) 
		(9.6,-.75)  --++  (-130:.2);
		
		\draw
		(9.4,.75) --++  (50:.2) 
		(9.4,.75)  --++  (-50:.2);
		
		\end{tikzpicture}}
	\caption{Folding of $\Delta_h=\mathrm{A}_5$ to $\Delta=\Delta_{h,\Cc}=\mathrm{C}_3$ or dually to $\Delta^{\!\!\vee}=\Delta_h^\Cc=\mathrm{B}_3$.}
	\label{Figure2}
\end{figure}

\subsection{Structure}
In Section \ref{s:folding}, we give a concise account of folding in Lie and singularity theory. 
In Section \ref{s:foldinghitchin} we first review Hitchin 
systems and generic Hitchin fibers by introducing cameral curves.
	Then we consider the folding of cameral curves and prove Theorem \ref{thm:intro1}.
	In Section \ref{s:crepantresol} we briefly recall the construction of the families 
	$\mathcal{X}\to \Bb(\cu,G)$, determine their crepant resolutions, and prove Theorem \ref{thm:intro2}. 
Most of our findings are illustrated by a running example to guide the reader. 

\subsection{Notation}\label{notations}
\begin{itemize} 
\item 
$\Delta$: irreducible Dynkin diagram, denoted $\Delta_h$ if required to 
be simply-laced,
that is, `homogeneous', according to \cite[\S6.1]{Slo}.
\item
$R=R(\Delta)$:  root system $R\subset (V,(-,-))$ associated with $\Delta$, 
denoted $R_h\subset (V_h, (-,-))$ if $\Delta=\Delta_h$; we identify $\Delta$ 
with a choice of simple roots in $R$
\item
$\Delta^{\!\!\vee}$: irreducible Dynkin diagram of the coroot system $R^{\vee}$ associated to $R=R(\Delta)$;
we view $R^\vee$ as a root system in $V^\ast$, throughout.
\item 
($\Delta_h,\Cc)$: pair consisting of the irreducible simply-laced 
Dynkin diagram $\Delta_h$ and a cyclic subgroup $\Cc:=\langle a \rangle$ of the group $\mathrm{Aut}_D(\Delta_h)$ of Dynkin graph automorphisms\footnote{These are graph automorphisms $a$\label{graphauto} such that $(a(\alpha),\alpha)\in\{0,2\}$ for each root $\alpha\in R_h$. 
Note that with this terminology, Dynkin diagrams of type $\Delta_h=\mathrm{A_{2n}}$ have no 
non-trivial Dynkin graph automorphisms, $\mathrm{Aut}_D(\Delta_h)=\{\mathrm{id}\}$,
while  $\mathrm{Aut}(\Delta_h)\cong\Z_2$. 
Moreover, $\Cc\cong\{\id\}$ or $\Cc\cong\Z_2$ or $\Cc\cong\Z_3$.} 
of $\Delta_h$ such that $\Delta^{\!\!\vee}=\Delta_h^{\Cc}$, cf. Section \ref{ss:foldroot}.
\item
$\gfr=\gfr(\Delta)=\gfr(R)$: simple complex Lie algebra associated with the irreducible Dynkin diagram $\Delta$ and root system $R$, denoted $\gfr_h$ if $\Delta=\Delta_h$. 
\end{itemize}

\subsection{Acknowledgments}
Florian Beck thanks Lara Anderson and Laura Schaposnik for an invitation 
to the Simons Center for Geometry and Physics, Stony Brook, where part of 
this research was conducted, and kindly acknowledges the funding by the DFG Emmy Noether grant AL 1407/2-1.
During the preparation of this work, Ron Donagi was supported in part by NSF grant 
DMS 2001673 and by Simons HMS Collaboration grant \# 390287.
Katrin Wendland thanks Igor Krichever, Motohico Mulase and Kenji Ueno for very 
helpful discussions.
The authors thank the Simons Center for Geometry and Physics, Stony Brook,
as well as the Max-Planck-Institute for Mathematics, Bonn, for their hospitality, 
as some of the work presented here was carried out during our visits to these institutions.

\section{Folding in Lie and singularity theory}\label{s:folding}
In this section, we provide the relevant background for folding from a Lie theoretical perspective,
and we introduce  a running example for folding.
\subsection{Folding of root systems}\label{ss:foldroot}
Let $\Delta_h$ be any irreducible Dynkin diagram of type $\mathrm{ADE}$ and $R_h\subset V_h$ the corresponding root system in the Euclidean vector space $(V_h,(-,-))$.
We identify the nodes of $\Delta_h$ with a choice of simple roots in $R_h$. 
Let $\Cc=\langle a \rangle \subset \mathrm{Aut}(\Delta_h)$ be a subgroup generated by a 
Dynkin graph automorphism $a\in \mathrm{Aut}_D(\Delta_h)$ of finite order 
$|a|=\mathrm{ord}(a)$. 
The definition of Dynkin graph automorphisms 
(cf.~footnote \ref{graphauto}) implies, for example, that
$\Cc$ is trivial if $\Delta_h$ is of type $A_n$ with even $n$.
The $\Cc$-coinvariants in $R_h$ are defined by 
\begin{equation} 
	R_{h,\Cc}:=R_h/(1-a)R_h,\quad R_{h,\Cc}\subset V_{h,\Cc}\mbox{ with } V_{h,\Cc}:=V_h/(1-a)V_h.
\end{equation}
We denote the class of $\alpha\in R_h$ in $R_{h,\Cc}$ by $\alpha_O$.
The following is a standard result (see, for example, \cite[\S 10.3]{Springer}, 
\cite{Stembridge-Folding}):
\begin{lem}\label{lem:folding}
	The coinvariants $R_{h,\Cc}$ define an irreducible root system in $V_{h,\Cc}=V_h/(1-a)V_h$. 
	Its irreducible Dynkin diagram $\Delta=\Delta_{h,\Cc}$ is determined by the following table
	\begin{center}
		\begin{tabular}{l c l }
			\toprule
			$\Delta_h$ & $\mathrm{ord}(a)$ & $\Delta=\Delta_{h,\Cc}$ \\ \midrule \addlinespace
			$\mathrm{A}_{2n-1}$ & $2$ & $\mathrm{C}_n$ $(n\geq 2)$  \\
			$\mathrm{D}_{n+1}$ & $2$ & $\mathrm{B}_n$ $(n\geq 3)$  \\
			$\mathrm{D}_4$ & $3$ & $\mathrm{G}_2$ \\
			$\mathrm{E}_6$ & $2$ & $\mathrm{F}_{\!4}$. \\
			\bottomrule
			\end{tabular} 
	\end{center}
	In particular, this establishes a one-to-one correspondence between irreducible Dynkin diagrams $\Delta$ and pairs $(\Delta_h,\Cc)$ of irreducible $\mathrm{ADE}$-Dynkin diagrams and subgroups $\Cc=\langle a \rangle\subset \mathrm{Aut}(\Delta_h)$ generated by a Dynkin graph automorphism (see footnote \ref{graphauto} for the definition). 
	We say that $\Delta$ is obtained from $(\Delta_h,\Cc)$ through folding. 
\end{lem}

The dual version of folding replaces $\Cc$-coinvariants by $\Cc$-invariants and thereby interchanges short and long roots.
More precisely, let $O(\alpha)$ be the orbit of $\alpha\in R_h$ under the action of $\Cc=\langle a \rangle$. 
Then we define
\begin{equation}\label{eq:alphaO} 
\alpha^O:=\sum_{\alpha^\prime\in O(\alpha)} \alpha^\prime.
\end{equation}
Note that the sum does not involve multiplicities, so in particular, $a\cdot \alpha=\alpha$ implies $\alpha^O=\alpha$. 
By construction, we have 
\begin{equation}\label{eq:Ra}
R_h^\Cc=\{ \alpha^O ~|~\alpha\in R_h\} \subset V_h^\Cc.
\end{equation}

\begin{lem}\label{lem:foldingroots}
	If $\Delta=\Delta_{h,\Cc}$ as in Lemma \ref{lem:folding}, then $R_{h}^\Cc\subset V_h^\Cc$ is an irreducible root system with Dynkin diagram $\Delta^{\!\!\vee}$, that is, $\Delta^{\!\!\vee}=\Delta_h^\Cc$.
	Moreover, if for each $\alpha_h\in R_h$ we define $\alpha^\vee \in V_h^*$ by
\begin{equation}\label{dual}
\forall v\in V_h\colon\qquad \alpha^{\vee}(v) := 2\frac{ (\alpha,v) }{ (\alpha,\alpha) },
\end{equation}
then the map
	\begin{equation*}
	(R_{h,\Cc})^\vee \longrightarrow (R_h^\vee)^\Cc,\quad (\alpha_O)^\vee\mapsto (\alpha^\vee)^O
	\end{equation*}
	is an isomorphism of root systems in $(V_{h,\Cc})^* \cong (V_h^*)^\Cc$.
	In other words, $(\Delta_{h,\Cc})^\vee\cong(\Delta_h^{\!\!\vee})^\Cc$.
	\end{lem}
Since in our conventions, $\Delta_h$ is simply-laced, we have $\Delta_h^{\!\!\vee}\cong \Delta_h$, so Lemma \ref{lem:foldingroots} also implies $\Delta\cong(\Delta^\Cc_h)^\vee$.
\begin{ex}\label{Ex:roots}
As a running example throughout this work, we present the folding of $\Delta_h=\mathrm{A}_3$ to $\Delta=\mathrm{C}_2$.

On the level of root systems, for $\Delta_h=\mathrm{A}_3$ we choose a  basis $(\alpha_1,\,\alpha_2,\,\alpha_3)$
of simple roots  with 
$$
(\alpha_j,\alpha_j) = 2 \mbox{ for } j\in\{1,\,2,\,3\}, \qquad
(\alpha_1,\alpha_2) = (\alpha_2,\alpha_3) = -1, \qquad
(\alpha_1,\alpha_3) = 0.
$$
To fold $\Delta_h=\mathrm{A}_3$ to $\Delta=\mathrm{C}_2$, we must implement the
automorphism $a$ that is induced by $\alpha_1\leftrightarrow\alpha_3$.
Since $\Delta\cong (\Delta_h^\Cc)^{\vee}$, we may work with the basis 
$(\alpha_1^\vee,\,\alpha_2^\vee,\,\alpha_3^\vee)$ of $V_h^\ast$
which is dual to $(\alpha_1,\,\alpha_2,\,\alpha_3)$ under \eqref{dual}, such that $\Delta$
corresponds to a coroot system with simple basis $(\tfrac{1}{2}(\alpha_1^\vee+\alpha_3^\vee),\,\alpha_2^\vee)$.
\end{ex}

\subsection{Folding and Cartan subalgebras} \label{sec:foldingCartan}
As before, let $R=R_{h,\Cc}$ be a folded root system, i.e. 
$R\subset V$, and $R^\vee\subset V^\ast$ its coroot system with $V=V_{h,\Cc}$  and $V^\ast=(V_h^*)^\Cc$.
They have Dynkin diagrams $\Delta=\Delta_{h,\Cc}$ and 
$\Delta^{\!\!\vee}= \Delta_h^\Cc$, respectively.

We next work out the relation between the action of the Weyl group 
$W_h=W(R_{h}^\vee)$ of 
$R_h^\vee\subset V_h^*$ and
the Weyl group 
$W=W(R^\vee)$ of
$R^\vee\subset V^*$
on the Cartan tori 
$\tfr_h=V_h^\ast\otimes_{\R}\C$ and $\tfr=\tfr_h^\Cc$, 
respectively:
we may identify $W$ as a subgroup of $W_h$ and thereby 
$\tfr/W$ with the $\Cc$-invariants in
$\tfr_h/W_h$. This statement is known, see e.g. 
\cite[\S8, Remarks, p.\ 144]{Slo},
however we find it useful to present a self-contained proof. 
Indeed, first we note

\begin{prop}\label{p:foldingweyl}
Let $W_h^{\Cc}:=\{ w\in W_h~|~aw=wa\}\subset \mathrm{Aut}(V^*_h)$. 
Then the restriction to $V^*= (V_h^*)^\Cc\subset V_h^*$,
\begin{equation}\label{eq:resW}
 W_h^{\Cc}\longrightarrow \mathrm{Hom}(V^*, V_h^*), \quad w\mapsto w_{|V^*}
\end{equation}
induces an isomorphism of groups $W_h^\Cc\cong W$. 
\end{prop}

\begin{proof}
First observe that by construction, $w\in W_h^{\Cc}$ implies $w_{|V^*}\in \mathrm{Aut}(V^*)$, and $w(R^\vee)=R^\vee$ is immediate by \eqref{eq:alphaO} (also see \eqref{eq:weylfolding} below).
Since $R^\vee$ is a folded root system, 
$\mathrm{Aut}(R^\vee)=W \rtimes \mathrm{Aut}(\Delta^{\!\!\vee})=W$ and therefore $w_{|V^*}\in W$.
It remains to prove that $W_h^{\Cc}\longrightarrow W,\;w\mapsto w_{|V^*}$ is bijective.

\underline{Injectivity:}
Let $\rho_h=\tfrac{1}{2} \sum_{\alpha\in R_h^+} \alpha$ be the Weyl 
vector associated to the positive coroots $R_h^+\subset R_h^\vee$ and 
$\rho=\tfrac{1}{2}\sum_{\beta\in R^+} \beta$ be the Weyl vector associated to the 
positive coroots $R^+\subset R^\vee$. 
Since $R^\vee=(R_h^\vee)^{\Cc}$, by \eqref{eq:Ra} (replacing $R_h$ by $R_h^\vee$)
we have $\rho=\rho_h\in V_h$. 
If $w_{|V^\ast}=\mathrm{id}$ for some $w\in W_h^\Cc$, then 
$w(\rho_h)=w(\rho)=\rho=\rho_h$. 
Hence $w$ fixes a vector in the fundamental Weyl chamber associated to 
$R_h^+$and thus a fundamental domain for $W_h$, and therefore $w=\id$.  

\noindent
\underline{Surjectivity onto $W$:}
It suffices to prove that every simple reflection 
$s_\beta$, $\beta\in \Delta^{\!\!\vee}$, is in the image of the homomorphism \eqref{eq:resW}. 
Let $\beta=\alpha^O$ for $\alpha\in \Delta_h$ as in \eqref{eq:alphaO}. 
Then 
\begin{equation}\label{eq:weylfolding}
s_\beta= \widetilde s_\beta{}_{|V^\ast}\quad \mbox{with}\quad
\widetilde s_\beta:=\prod_{\alpha^\prime\in O(\alpha)} s_{\alpha^\prime},
\end{equation}
where we have used the very definition of Dynkin graph automorphisms
(footnote \ref{graphauto}).
Since $a$ permutes $O(\alpha)$ and $a s_{\alpha^\prime} a^{-1} = s_{a\alpha^\prime}$
for every $\alpha^\prime\in O(\alpha)$, we have $\widetilde s_\beta\in W_h^\Cc$, and surjectivity onto $W$ of 
our map follows. 
\end{proof}

We may thus indeed realize $W$ naturally as a subgroup of $W_h$. For its
action on the regular elements of $\tfr=\tfr_h^\Cc$, we note
\begin{lem}\label{regact}
Assume that $t\in\tfr$ and $w\in W_h$ with $wt\in\tfr$. Then the orbits of 
$t$ and $wt$ under $W$ agree, $W(wt)=W(t)$. In particular, if 
$t$ belongs to the set $\tfr^\circ$ of regular elements of $\tfr$, then $w\in W$.
\end{lem}
\begin{proof}
First note that $\tfr^\circ\subset\tfr_h^\circ$.
Indeed, if $t\in\tfr$ but $t\notin\tfr_h^\circ$, then $\alpha(t)=0$ 
for some $\alpha\in\Delta_h$.
Now $a(t)=t$ because $\tfr=\tfr^\Cc_h$, and since $a\in\mbox{Aut}(\Delta_{h})$, for $\beta:=\alpha_O$, we conclude
$\beta(t)=0$ with $\beta\in\Delta_{h,\Cc}$. 
In other words, $t\notin\tfr^\circ$.

Now let $t\in \tfr^\circ$, so by the above, $t\in \tfr_h^\circ$.
Then $w\in W$ if and only if $W(wt)=W(t)$.
To show that indeed $w\in W$, choose $w_1,\,w_2\in W$
such that $t^\prime := w_1 t$ and $w_2 w t$ both belong to  the 
fundamental Weyl chamber in $\tfr$ corresponding to our choice of simple roots of $R$.  
By construction and assumption, with $w^\prime := w_2  w (w_1)^{-1}$, both $t^\prime$ and $w^\prime t^\prime$ belong to
the same fundamental Weyl 
chamber in $\tfr_h$ with $w^\prime\in W_h$, hence $w^\prime=\mathrm{id}$ and $w= (w_2)^{-1} w_1\in W$ follows.

If $t\notin\tfr^\circ$ is non-zero, then we work with the root subsystem 
$R(t)$ of $R$ defined by all roots $\alpha\in R$ such that $\alpha(t)\neq 0$.
Then $t$ is regular with respect to $R(t)$ and the above argument applies.
If $t=0$, there is nothing to show.
\end{proof}

According to classical results by Chevalley\footnote{These hold for any complex Lie algebra
that corresponds to a complex reductive algebraic group $H$.\label{Chevalley}} \cite{Che}, 
explained, for example, in 
\cite[Section 23]{Humphreys},
the coordinate ring $\C[\tfr]^{W}$ of $\tfr/W$ is freely generated by polynomials
$\chi_1,\ldots,\chi_r$ of degrees $d_1,\ldots,d_r$, respectively, where $r$ denotes
the rank of $\gfr$ and $\epsilon_j=d_j-1$, $j\in\{1,\ldots,r\}$, are the exponents of this Lie
algebra, and similarly for $\gfr_h$. This induces natural $\C^\ast$-actions on $\tfr/W$
and $\tfr_h/W_h$,  as well as non-canonical 
isomorphisms $\tfr/W\cong\C^r$, $\tfr_h/W_h\cong\C^{r_h}$ 
as $\C^*$-spaces with weights $d_1,\dots, d_r$.

\begin{cor}\label{cor:tfr}
\hfill\\[-5pt]
	\begin{enumerate}[label=\roman*)]
		\item The reflection hyperplanes $\tfr^{\alpha^O}\subset \tfr$, $\alpha^O\in 
		R^\Cc_h=R^\vee$, 
		are given by 
				\begin{equation}\label{eq:tfrO}
				\tfr^{\alpha^O}=\bigcap_{\alpha^\prime\in O(\alpha)} \tfr_h^{\alpha^\prime} \cap \tfr\quad \subset\quad \tfr_h.
				\end{equation}
	\item The inclusion $\tfr=\tfr_h^\Cc\hookrightarrow \tfr_h$ induces the isomorphism
			\begin{equation*}
			\tfr/W\cong (\tfr_h/W_h)^\Cc
			\end{equation*}
			of $\C^*$-spaces.
	\end{enumerate}
\end{cor}
\begin{proof}
\hfill\\[-10pt]
	\begin{enumerate}[label=\roman*)]
		\item 
		This is a direct consequence of \eqref{eq:weylfolding}. 
		\item 
		As stated before, the claim follows from Slodowy's work,
		see \cite[\S8, Corollary and Remarks, p.\ 144]{Slo}. But it also follows from 
		what was shown in this section, so far.
Indeed, $\tfr\hookrightarrow\tfr_h\longrightarrow \tfr_h/W_h$ factorizes over
$\tfr/W$ since $W\subset W_h$ by Proposition \ref{p:foldingweyl}. Moreover, since
$W_h$ is a normal subgroup of $\mathrm{Aut}(\Delta_h)$, $\Cc$ acts naturally on
$\tfr_h/W_h$ such that $\tfr/W$ maps to $(\tfr_h/W_h)^\Cc$.  
The $\Cc$-action commutes
with the $\C^*$-action by construction, so this is a homomorphism of $\C^*$-spaces.
We need to show that it is injective and surjective.

\underline{Injectivity:}
If $t,\, t^\prime\in\tfr$ with $t^\prime=w t$ for some $w\in W_h$, 
we need to show that $t^\prime=w^\prime t$ for some $w^\prime\in W$. 
But this is immediate from  Lemma \ref{regact}.

\underline{Surjectivity:}
Assume that $t\in\tfr_h$ and $at=wt$ for some $w\in W_h$, that is,
$t$ represents a class in $(\tfr_h/W_h)^\Cc$. 
Since $W_h$ is a normal subgroup of $\mathrm{Aut}(\Delta_h)$, $\Cc$ acts on the orbit $W_h(t)$ of $t$ under $W_h$.
On the other hand, $W_h(t)$ intersects the closure of the fundamental Weyl chamber in a unique element $w^\prime t\in W_h(t)$, $w^\prime\in W_h$. 
Hence $aw^\prime t=w^\prime t$, and $w^\prime t\in \tfr$ represents a class in $\tfr/W$ which maps to the class of $t$  in $\tfr_h/W_h$.
\end{enumerate}
\end{proof}
It is important to keep in mind that the simple roots in the root system $\Delta_h$ 
form a basis of $\tfr_h^\ast$; descending to $\Cc$-invariants in $\tfr_h/W_h$
as in Corollary \ref{cor:tfr} above thus corresponds to descending to 
coinvariants on the level of root systems, $\Delta=\Delta_{h,\Cc}$. This justifies our
choice of notations, where $\Delta^{\!\!\vee} = \Delta_h^\Cc$ in contrast
to the conventions in Slodowy's work \cite{Slo}, cf.\ footnote \ref{footnote:Slo}
on page \pageref{footnote:Slo} below.
\begin{ex}\label{runLiealgebras}
The Lie algebras encoded by the Dynkin data $\Delta_h=\mathrm{A}_3$ and
$\Delta=\mathrm{C}_2$ are $\mathfrak{g}_h=\mathfrak{sl}_4(\C)$ and 
$\mathfrak{g}=\mathfrak{sp}_4(\C)$, respectively.
We may thus use
\begin{eqnarray*}
\mathfrak{g}_h &=& \left\{ A\in \mbox{Mat}_{4\times 4}(\C) \mid \mbox{tr}(A) =0 \right\}, \\
\mathfrak{g} &=& \left\{ \begin{pmatrix}  -d^T & b\\ c&  d \end{pmatrix}\bigg|
 b,\, c,\, d \in \mbox{Mat}_{2\times 2}(\C),\;
b=b^T,\; c=c^T \right\}.
\end{eqnarray*}
As our Cartan subalgebras, we choose the subalgebras of diagonal matrices in each case,
$$
\mathfrak{t}_h= \left\{ \mbox{diag}(u,\,v,\,w,\,-u-v-w)\mid u,\,v,\,w\in\C \right\}, \quad
\mathfrak{t} = \left\{ \mbox{diag}(u,\,v,\,-u,\,-v)\mid u,\,v\in\C \right\}.
$$
For later convenience we also choose simple coroots for $\mathfrak{g}_h=\mathfrak{sl}_4(\C)$, namely
$$
\alpha_1^\vee := \mbox{diag}(1,\,-1,\,0,\,0), \quad
\alpha_2^\vee := \mbox{diag}(0,\,1,\,-1,\,0), \quad
\alpha_3^\vee := \mbox{diag}(0,\,0,\,1,\,-1),
$$
while 
for $\mathfrak{g}=\mathfrak{sp}_4(\C)$ we use
$$
\beta_1^\vee :=\tfrac{1}{2} \mbox{diag}(-1,\,-1,\,1,\,1), \quad
\beta_2^\vee :=   \mbox{diag}(0,\,1,\,0,\,-1).
$$
The generators $s_j$ of the Weyl group $W_h$ corresponding to our choices
of simple roots $\alpha_j^\vee$ act by transpositions $(j,\,j+1)$, exchanging the $j^{\mbox{\tiny th}}$ and
the $(j+1)^{\mbox{\tiny th}}$ diagonal entry of any matrix $t\in\tfr_h$, 
while $\beta_1^\vee$ corresponds to $s_1s_2s_3s_2s_1s_2=(1,\, 4)(2,\,3)$ and $\beta_2^\vee$ corresponds to $s_2s_3s_2=(2,\,4)$.
As remarked in Example \ref{Ex:roots}, the generator $a$ of the $\Cc$-action for $\Delta_h$ is induced by 
$\alpha_1^\vee\leftrightarrow\alpha_3^\vee$.
Let us explicitly verify $\tfr/W\cong (\tfr_h/W_h)^\Cc$ as $\C^*$-spaces in this example.
To describe $\tfr/W$, we work with the fundamental Weyl chamber in $\tfr$ corresponding
to the simple root basis $(\beta_1^\vee,\, \beta_2^\vee)$, and analogously for
$\tfr_h/W_h$. Then the isomorphism $\tfr\cong \tfr_h^\Cc$ given by
\begin{equation}\label{eq:beta-alpha}
\beta_1^\vee\mapsto \tfrac{1}{2}(\alpha_1^\vee+\alpha_3^\vee),\qquad \beta_2^\vee\mapsto \alpha_2^\vee
\end{equation}
induces an isomorphism $\tfr/W \cong (\tfr_h/W_h)^\Cc$ which is $\C^*$-equivariant. 
Indeed, the coordinate ring $\C[\tfr/W]= \C[\tfr]^{W}$ is generated by the even 
elementary symmetric polynomials $\sigma_j$, $j\in \{ 2,4\},$ of degree $j$ 
in the diagonal entries of each element of $\tfr$.
This yields natural coordinates $\xi=(\sigma_2,\sigma_4)$ on $\tfr/W$.
On $\tfr_h/W_h$ we similarly have coordinates $\xi_h=(\sigma_2,\sigma_3,\sigma_4)$
where $\sigma_3$ vanishes on $(\tfr_h/W_h)^\Cc$. We thus may use
$\check\xi_h=(\sigma_2,\sigma_4)$ as coordinates on $(\tfr_h/W_h)^\Cc$ and check
$$
\check\xi_h( u(\alpha_1^\vee +\alpha_3^\vee) + (u+v) \alpha_2^\vee ) 
= (-u^2-v^2, u^2v^2) = \xi(2u\beta_1^\vee + (u+v) \beta_2^\vee ).
$$
Hence the isomorphism $\check{\xi}_h\mapsto \xi$ of coordinate 
rings yields the isomorphism $\tfr/W\cong (\tfr_h/W_h)^\Cc$ 
induced by \eqref{eq:beta-alpha}, which is therefore $\C^*$-equivariant.
\end{ex}
\subsection{Folding of simple complex Lie algebras}\label{ss:foldingLAs}
Let $\gfr_h=\gfr(\Delta_h)$ be the simple complex Lie algebra associated with $\Delta_h$. 
If $G_h$ denotes the simple adjoint complex Lie group of $\gfr_h$, then we have 
the short exact sequence
\begin{equation}\label{eq:exactseq}
\begin{tikzcd}
 1 \ar[r] & G_h \ar[r] & \mathrm{Aut}(\gfr_h) \ar[r] & \mathrm{Aut}(\Delta_h) \ar[r] & 1, 
\end{tikzcd}
\end{equation}
see \cite[\S8.8]{Slo}.
This short exact sequence splits: 
let $\alpha_i\in \Delta_h$, $i\in\{1,\dots, r\}$, denote a choice of simple roots, and choose 
a Chevalley basis $(e_\alpha,\, \alpha_i^\vee~|~\alpha \in R_h, \, 
i\in\{1,\dots, r\})$ of $\gfr_h$. 
This means that for $\alpha\in R_h$, $e_\alpha$ forms a basis of the root space $\gfr_{h,\alpha}$, and $(\alpha_1^\vee,\,\ldots,\,\alpha_r^\vee)$ is a simple basis of $\tfr_h$, with normalizations as follows. 
For every $\alpha\in R_h$, $[e_\alpha,e_{-\alpha}]=h_\alpha$ with
$\alpha(h_\alpha)=2$. If $\alpha,\,\beta,\,\alpha+\beta\in R_h$, then
$$
[e_\alpha, e_\beta]= c_{\alpha,\beta}\, e_{\alpha+\beta}\qquad
\mbox{ with } \qquad c_{-\alpha,-\beta}=-c_{\alpha,\beta}\;\in\Z.
$$
These normalizations ensure that for any $a\in \mathrm{Aut}(\Delta_h)$ we may
define, by abuse of notation, the automorphism
$a\colon\gfr_h\longrightarrow \gfr_h$ by
\begin{equation}\label{splitaction}
a(e_\alpha):=e_{a\cdot \alpha}, \quad a(\alpha_i^\vee):= (a\cdot\alpha_i)^\vee. 
\end{equation}
Note that this splitting is compatible with the $\Cc$-action
on $\tfr_h$ already used in Section \ref{sec:foldingCartan}.
\begin{prop}\label{prop:isoinv}
Let $\gfr_h=\gfr(\Delta_h)$ be the simple complex Lie algebra associated to the irreducible Dynkin 
diagram $\Delta_h$ and $\Cc=\langle a \rangle$ the
subgroup of $\mathrm{Aut}(\Delta_h)$ generated by the Dynkin graph automorphism $a$. 
Further let $\Cc\hookrightarrow \mathrm{Aut}(\gfr_h)$ be determined by the
above splitting of \eqref{eq:exactseq}. 
Then 
\begin{equation*}\label{eq:isoinv}
\gfr_h^\Cc\cong \gfr(\Delta_{h,\Cc}).
\end{equation*} 
\end{prop}
\begin{proof}
Set $\gfr:=\gfr^{\Cc}_h$, $\tfr:=\tfr_h^{\Cc}$ and let 
\begin{equation}\label{eq:projectiongfr}
 p\colon\gfr_h \longrightarrow \gfr, \quad \xi \mapsto \tfrac{1}{|a|} \sum_{k=1}^{|a|} a^k(\xi)
\end{equation}
be the (vector space) projection.
Firstly note that $a$ is a Lie algebra homomorphism, which implies that the Killing form $\kappa_h$ on $\gfr_h$ restricts to a non-degenerate form  $\kappa$, the Killing form of $\gfr$. 
Thus by \cite[\S 5.1]{Humphreys}, $\gfr$ is semisimple.

We next show that $\tfr$ is a Cartan subalgebra.
	Since $\C$ has characteristic zero and $\gfr$ is semisimple, by \cite[\S 15.3]{Humphreys} it suffices to show that $\tfr\subset \gfr$ is a maximal toral subalgebra. 
	That $\tfr$ is toral follows immediately from the fact that $a$ is a Lie algebra homomorphism.
	Since 
	$$
	\dim(\tfr)=\mathrm{rk}(R^\vee)=\dim(\tfr_h^\Cc),\qquad R^\vee=R_h^\Cc,
	$$	
$\tfr$ is also maximal for dimensional reasons.
 
By the duality between invariants and coinvariants, we have
\begin{equation}
(\tfr_h^\Cc)^*\cong (\tfr_h^*)_{\Cc}. 
\end{equation}
In particular, $R_{h,\Cc}\hookrightarrow \tfr^*$. 

Next we claim that\footnote{Note that $\alpha_O=(a\cdot \alpha)_O$ for all $\alpha\in R_h$.}
\begin{equation}\label{eq:rootspace}
 \gfr= \tfr\oplus \bigoplus_{\alpha_O\in R_{h,\Cc}}  \gfr_{\alpha_O}
\end{equation}
is a root space decomposition of $\gfr$. 
Indeed, let $\eta=p(\xi)$ with $\xi\in \gfr_{h,\alpha}$, 
$s=p(t)$ with $t\in \tfr$ and compute 
\begin{align*}
 [s,\eta] =& \tfrac{1}{|a|^2} \sum_{k,l} [a^k(t), a^l(\xi)] \\
 =&  \tfrac{1}{|a|^2} \sum_{k,l} a^l ([a^{k-l}(t), \xi]) \\
 =& \tfrac{1}{|a|^2} \sum_{k,l} \alpha(a^{k-l}(t)) a^l(\xi) \\
 =& \tfrac{1}{|a|}\sum_l \alpha(s) a^l(\xi) 
 = \alpha_O(s) \eta. 
\end{align*}
It follows that $p$ surjects $\gfr_{h,\alpha}$ onto $\gfr_{\alpha_O}$. 
In particular, this shows \eqref{eq:rootspace}, thus concluding the proof. 
\end{proof}
The previous result combined with Corollary \ref{cor:tfr} yields the commutative diagram
\begin{equation}\label{eq:relationadjquotients}
\begin{tikzcd}
\gfr \ar[r, hookrightarrow]  \ar[d, "\chi"'] & \gfr_h \ar[d,"\chi_h"] \\
\tfr/W \ar[r, hookrightarrow] & \tfr_h/W_h. 
\end{tikzcd}
\end{equation}
Here $\chi$ and $\chi_h$ are the respective adjoint quotients. 

\begin{ex}\label{runchis}
For  $\gfr_h=\mathfrak{sl}_4(\C)$ and 
$\gfr=\mathfrak{sp}_4(\C)$ as in Example \ref{runLiealgebras}, we lift the action of $a$ to $\gfr_h$ by
\begin{equation}\label{Clift}
A
\longmapsto 
\left(\begin{matrix} \mathbf{1}&\mathbf{0}\\ \mathbf{0}&-\mathbf{1}\end{matrix}\right)
A^{\widetilde T} 
\left(\begin{matrix} -\mathbf{1}&\mathbf{0}\\\mathbf{0}&\mathbf{1}\end{matrix}\right),\qquad
\end{equation}
where $\mathbf{1}$ denotes the identity matrix in $\mbox{Mat}_{2\times2}(\C)$ and
$A^{\widetilde T}$ is obtained from $A$ by reflecting in the northeast-southwest 
diagonal. One immediately checks that this induces the action $\alpha_1^\vee\leftrightarrow\alpha_3^\vee$ on $\tfr_h$. Moreover,
$$
\mbox{with }\quad \mathbf{N}:=\left(\begin{matrix}0&1\\0&0\end{matrix}\right),\qquad
e_{\alpha_1} := 
\left(\begin{matrix} \mathbf{N}&\mathbf{0}\\ \mathbf{0}&\mathbf{1}\end{matrix}\right),\quad
e_{\alpha_2} := 
\left(\begin{matrix} \mathbf{0}&\mathbf{N}^T\\ \mathbf{0}&\mathbf{1}\end{matrix}\right),\quad
e_{\alpha_3} := 
\left(\begin{matrix} \mathbf{0}&\mathbf{0}\\ \mathbf{0}&-\mathbf{N}\end{matrix}\right)
$$
generate a Chevalley basis, where $e_{-\alpha}=e_\alpha^T$ for all $\alpha\in R_h$.
On this basis, \eqref{Clift} implements the action \eqref{splitaction}. We then have
$$
\gfr_h^\Cc = \left\{\left. \left(\begin{matrix}- d^{\widetilde T}& b\\ c&d\end{matrix}\right)
\right| b,\,c,\,d\in \mbox{Mat}_{2\times2}(\C),\, b=b^{\widetilde T},\, c=c^{\widetilde T} \right\}.
$$
Using the description of $\gfr$ from Example \ref{runLiealgebras}, we thus have
$\gfr_h^\Cc\cong\gfr$ under the automorphism which permutes rows and columns
by the transposition $(1,2)$. 
We continue to use the notations introduced in Example \ref{runLiealgebras}.
Then \eqref{eq:relationadjquotients} is obtained from $\chi\colon\gfr\longrightarrow\tfr/W$,
$\chi_h\colon\gfr_h\longrightarrow\tfr_h/W_h$, where 
\begin{equation}\label{chisexterior}
\xi\circ\chi(A) = (\mbox{tr}(\Lambda^2A),\, \mbox{tr}(\Lambda^4A)),\quad
\xi_h\circ\chi_h(A) = (\mbox{tr}(\Lambda^2A),\, \mbox{tr}(\Lambda^3A),\, \mbox{tr}(\Lambda^4A)),
\end{equation}
with $\mbox{tr}(\Lambda^3A)=0$ for all $A\in\mathfrak{sp}_4(\C)$.
\end{ex}

\subsection{Folding of simple complex Lie groups}\label{ss:foldingsimpleliegroups}
Let $G_h$ be the simple \emph{adjoint} complex Lie group with irreducible $\mathrm{ADE}$-Dynkin diagram 
$\Delta_h$ and $\Cc=\langle a \rangle$ as before. 
In the following we fix a maximal torus $T_{h}\subset G_{h}$. 
Its character and its cocharacter lattice are given by
\begin{align}
\bLambda^{\!\!\vee}_h&=\mathrm{Hom}(T_{h},\C^*)\cong\langle R_h\rangle_{\Z}, \label{characterdef}\\ 
\bLambda_h& = \mathrm{Hom}(\C^*,T_{h}).\label{cocharacterdef}
\end{align}
Here, $\bLambda^{\!\!\vee}_h\cong\langle R_h\rangle_{\Z}$
follows since $G_{h}$ is of adjoint type. 
This further implies that $\mathrm{Aut}(G_h)\cong\mathrm{Aut}(\gfr_h)$,
and thus our choice of splitting of the sequence 
\eqref{eq:exactseq} yields $\Cc\hookrightarrow \mathrm{Aut}(G_h)$. 

\begin{prop}\label{foldedgroup}\cite[Proposition 10.3.5]{Springer}
	The identity component $(G_h^\Cc)^\circ$ of the subgroup $G_h^\Cc\subset G_h$ 
	fixed by the automorphism group $\Cc\subset \mathrm{Aut}(G_h)$ is the simple adjoint complex Lie group 
	$G$ with Dynkin diagram $\Delta_{h,\Cc}$. 
	In particular, its character and cocharacter lattices satisfy 
	\begin{equation}\label{foldedcharacter}
	\bLambda^{\!\!\vee}=(\bLambda_h^{\!\!\vee})_{\Cc},\quad \bLambda=\bLambda_h^\Cc.
	\end{equation}
\end{prop}

\begin{ex}\label{rungroupfold}
For Dynkin type $\Delta_h=A_3$, the adjoint form is $G_h=PSL_4(\C)$, while
$SL_4(\C)$ is the simply connected form. On the other hand, the adjoint form for
$\Delta=C_2$ is $G=Sp_4(\C)\slash\{\pm1\}=PSp_4(\C)$, while $Sp_4(\C)$ is the simply
connected form. We lift the action of $\Cc$ used in Example \ref{runchis} to $G_h$,
and  we obtain
$$
\forall\, A\in SL_4(\C)\colon\quad
A\mapsto - \widetilde J (A^T)^{-1} \widetilde J, \quad\mbox{where } 
\widetilde J
:=\left(\begin{matrix}\mathbf{0}&\mathbf{Q}\\-\mathbf{Q}&\mathbf{0}\end{matrix}\right)
=- \widetilde{J}^{-1},\;
\mathbf{Q}:= \left(\begin{matrix}0&1\\1&0\end{matrix}\right).
$$
Here, 
$$
\widetilde J = P J P
\quad\mbox{ with }\quad 
P= \left(\begin{matrix}\mathbf{Q}&\mathbf{0}\\\mathbf{0}&\mathbf{1}\end{matrix}\right),\;
J =\left(\begin{matrix}\mathbf{0}&\mathbf{1}\\-\mathbf{1}&\mathbf{0}\end{matrix}\right),
$$
so $J$ is the standard symplectic matrix. 
It follows that $(G_h)^\Cc$ arises from $G$ by conjugation
with $P$. 
Hence $(G_h)^\Cc\cong G$ holds.
\end{ex}

\subsection{Folding of $\mathrm{ADE}$ singularities and Slodowy slices} \label{sec:Slodowy}
The construction of quasi-projective Calabi--Yau threefolds of \cite{DDP} starts from semi-universal $\C^\ast$-deformations
of $\mathrm{ADE}$ surface singularities. Following \cite{Slo}, there is
an elegant description of the latter in terms of the associated simple complex Lie algebras, which is compatible
with folding. 
In this section, we recall that construction, mainly following the original work \cite{Slo}.
 
\begin{dfn} 
Let $\Delta=\Delta_{h,\Cc}$ denote\footnote{This definition goes back to Slodowy \cite[\S6.2]{Slo}, 
who however uses the dual notion, 
i.e. $\Cc$-invariants instead of $\Cc$-coinvariants to label the singularities. Our conventions are better 
adapted to the folding 
of simple complex Lie algebras and Hitchin systems.\label{footnote:Slo}} an irreducible Dynkin 
diagram with $\Delta_h$ an 
irreducible Dynkin diagram of type $\mathrm{ADE}$ and $\Cc\subset \mathrm{Aut}(\Delta_h)$ as before.
A $\Delta$-singularity is a pair $(Y,H)$ such that
\begin{itemize}
\item
$Y=(Y,0)$ is (a germ of) an isolated surface singularity of type $\Delta_h$; in particular, the dual of the resolution 
graph of the minimal resolution $\widetilde{Y}\longrightarrow Y$ coincides with $\Delta_h$, 
\item 
$H\subset \mathrm{Aut}(Y)$ is a finite subgroup with $H\cong \Cc$ which acts freely on $Y-\{0\}$, 
\item 
the induced action of $H$ on the dual of the resolution graph coincides with the $\Cc$-action on $\Delta_h$. 
\end{itemize}
If $H=\{\id\}$, then we simply write $Y$ instead of $(Y,\{\mathrm{id} \})$.
\end{dfn}
Every $\Delta$-singularity is quasi-homogeneous, i.e.~there is a natural $\C^*$-action on each $\Delta$-singularity $(Y,H)$. 
Moreover, every $\Delta$-singularity $(Y,H)$ admits a semi-universal $\C^*$-deformation according to \cite[\S2.4-2.7]{Slo}. 
By this, following \cite[Definition 2.6]{Slo}, we mean a (formal) $(\Cc\times\C^\ast)$-equivariant deformation $\zeta\colon\mathcal Y\to B$ of $Y$ with trivial action
of $\Cc$ on the base such that every other deformation with these properties allows a $(\Cc\times\C^\ast)$-equivariant morphism to $\zeta$ whose differential on the base is uniquely determined.

If the isolated surface singularity $Y=(Y,0)$ is defined by a 
quasi-homogeneous polynomial $f\in\C[x,\,y,\,z]$ of weights\footnote{Clearly, 
these weights are not unique and we fix one choice here. As we shall see, 
cf.~Remark \ref{rem:Cstar}, the natural choice of coprime weights is in general 
not compatible with the Lie algebraic description of $\Delta$-singularities.}
$(\mathsf{w}_x, \mathsf{w}_y, \mathsf{w}_z)$ and degree 
$\deg(f)=\mathsf{w}_x+\mathsf{w}_y+\mathsf{w}_z$, 
then the work of Schlessinger \cite{Sch68}, Elkik \cite{Elkik-Algebraization} and Rim \cite[4.14]{Ri72} gives a direct way to construct a semi-universal 
$\C^*$-deformation of $Y$.
The construction uses the Jacobian ring $J_f:= \C[x,y,z]/(\partial_x f,\partial_y f, \partial_z f)$ 
of $f$ as follows.
In our situation, $J_f$ is a finite dimensional vector space, $J_f\cong\C^r$.
If $(g_1,\,\ldots,\,g_r)$ with quasi-homogeneous $g_j\in\C[x,y,z]$ for $j\in\{1,\,\ldots,\,r\}$ 
represents a $\C$-basis of $J_f$, then 
\begin{eqnarray*}	
 \mathcal{Y}_h:=\left\{((x,y,z),b)\in \C^3\times\C^r ~|~ 
f(x,y,z)+\smash{\sum_{j=1}^r}\vphantom{\sum} b_j g_j(x,y,z)=0 \right\}, &
& \hspace*{-1em} B_h:=\C^r, \\	
\zeta\colon\;\mathcal{Y}_h\longrightarrow B_h, &&
		((x,y,z),b)\mapsto b
\end{eqnarray*}
defines a semi-universal $\C^\ast$-deformation of $Y$, cf. \cite[Theorem 2.4]{Slo}. 
Note that $B_h\cong J_f$ by construction.
The $\C^\ast$-action on $Y$ is naturally extended to $\mathcal Y_h$ 
by assigning the weight $\deg(f) -\deg(g_j)$ to $b_j$ for $j\in\{1,\,\ldots,\,r\}$.
Then $\mathcal{Y}_h \to B_h$ is a semi-universal $\C^*$-deformation of 
$Y$ by \cite[Theorem 2.5]{Slo} and its proof.
If $(Y,H)$ is any $\Delta$-singularity, then the $H$-action on $Y$ extends 
to an $H$-action on $\mathcal{Y}_h$ and $B_h$. 
	The restriction $\zeta\colon \zeta^{-1}(B_h^H)\to B_h^H$ is 
	then a semi-universal $\C^*$-deformation of $(Y,H)$, see \cite[\S2.6]{Slo}.

Let us illustrate this construction by carrying it out for our running example.
\begin{ex}\label{ex:A3-sing}
	We consider $\mathrm{C}_2=\mathrm{A}_{3,\Cc}$ for $\Cc=\Z/2\Z$.
	The $\mathrm{A}_3$-surface singularity is given by
	\begin{equation}
	Y=\{ (x,y,z)\in \C^3~|~x^4-yz=0\}. 
	\end{equation}
	The defining polynomial $f(x,y,z):= x^4-yz$ is quasi-homogeneous with respect to the $\C^*$-action $(\lambda, (x,y,z))\mapsto(\lambda x, \lambda^2y, \lambda^2z)$.
	Consider the group $H\cong \Cc$ generated by the automorphism $(x,y,z)\mapsto (-x,z,y)$. 
	Then $(Y,H)$ is a $\Delta$-singularity with $\Delta=\mathrm{C}_2$, as one checks by an explicit calculation, see \cite[p.\ 77]{Slo}. 
	
	The Jacobian ring for this singularity is $J_f=\C[x,y,z]/(x^3,\, y,\, z)$, so $(x^2,\, x,\, 1)$ represents a quasi-homogeneous basis. 
	Hence a semi-universal $\C^*$-deformation of the $\mathrm{A}_3$-singularity is given by
		\begin{eqnarray}\label{YhA3}	
		\mathcal{Y}_h=\{(x,y,z,b_2,b_3,b_4)\in \C^6 ~|~ x^4-yz  + b_2x^2  + b_3x  + b_4=0 \} &\longrightarrow& \C^3, \nonumber\\
		(x,y,z,b_2,b_3,b_4)&\mapsto& (b_2,b_3,b_4)
		\end{eqnarray}
	with $\C^\ast$-action
			\begin{equation}\label{eq:suC}
	\forall\,\lambda\in \C^*\colon\qquad
	\lambda\cdot(x,y,z,b_2,b_3,b_4)=(\lambda x,\lambda^2 y, \lambda^2 z,\lambda^2 b_2,\lambda^3 b_3, \lambda^4 b_4).
	\end{equation}
	The action of $H$ is  extended to $\mathcal Y_h$ by
	\begin{equation}\label{eq:suHaction}
	(x,y,z,b_2,b_3,b_4)\longmapsto (-x,z,y,b_2,-b_3,b_4).
	\end{equation}
		The preimage under $\mathcal{Y}_h\to \C^4$ of the $H$-fixed point locus in $\C^4$ is a semi-universal $\C^*$-deformation of the $\mathrm{C}_2$-singularity $(Y,H)$, namely
		\begin{eqnarray*}
		\mathcal{Y}=\{x,y,z,b_2,b_4)\in \C^5 ~|~ 
		x^4-yz  +  b_2x^2 +  b_4=0 \} 
		&\to& \C^2, \\
		(x,y,z,b_2,b_4)&\mapsto& (b_2,b_4).
		\end{eqnarray*}
	The $\C^*$- and $H$-action are the restrictions of \eqref{eq:suC} and \eqref{eq:suHaction} respectively. 
	By construction, the $H$-action is trivial on the base, so $H$ acts on all fibers.
	\end{ex}
The direct construction of semi-universal $\C^*$-deformations is useful in explicit 
computations, see Examples \ref{ex:A3cys} and \ref{ex:g2}.
The Lie-theoretic construction due to Brieskorn \cite{Bri} (for $\Delta_h$-singularities) 
and Slodowy \cite[\S8]{Slo} (for all $\Delta$-singularities) is very convenient for our study of the relation between Calabi--Yau and Hitchin integrable systems, cf. Section \ref{s:crepantresol}.
We emphasize that with appropriate choices of $\C^\ast$-actions (cf.~Remark \ref{rem:Cstar}) 
both constructions give isomorphic $\C^*$-deformations by semi-universality, 
see \cite[Proposition 2.9]{Sch68}. 
In Appendix \ref{app:sh} we work out an explicit isomorphism for our running example.

To review Slodowy's construction, let $\gfr=\gfr(\Delta)$ denote the simple complex Lie algebra determined 
by $\Delta$, and let $x\in \gfr$ be a subregular nilpotent element. 
Choose an $\mathfrak{sl}_2$-triple $(x,y,h)$
for $x$ with semisimple element $h\in \mathfrak{g}$ such that $[h,x]=2x,\, [h,y]=-2y,\, [x,y]=h$, so 
$\langle x,y,h\rangle_\C\cong\mathfrak{sl}(2,\C)$.
Then a so-called Slodowy slice through $x$ is given by
\begin{equation*}
 S:=x+\ker \mathrm{ad}_y\subset \gfr,
\end{equation*}
with the restriction $\sigma:= \chi_{|S}$ of the adjoint quotient,
\[
\sigma\colon S\longrightarrow \tfr/W.
\]
We note that any two nilpotent subregular $x\in\gfr$ are conjugate
under the adjoint action of $G$ \cite[\S 3.10]{Steinberg}, as are any two
$\mathfrak{sl}_2$-triples that contain $x$ \cite[Corollary 3.6]{Kostant}.
The Slodowy slice $S$ carries actions by our cyclic group $\Cc$ and by $\C^\ast$, 
which commute with each other, as follows.
To obtain the $\C^\ast$-action, let 
$\exp\colon\gfr\longrightarrow G$ denote the exponential map 
from the Lie algebra to the Lie group stemming from the construction of 
one-parameter subgroups.
Observe that $\mathrm{ad}_{th}(x)=2tx$ and $\lambda=e^t\in\C^\ast$ yield $\mathrm{Ad}_{\exp(th)}(x)=\exp(\mathrm{ad}_{th})(x)=\lambda^2x$ because $[h,x]=2x$. 
Thus there is a $\C^*$-action on $S$ given by
\begin{equation}\label{CstaronS}
\forall \lambda=e^t\in\C^\ast,\; v\in S\colon\quad
\lambda\cdot v
=\lambda^2 
\mathrm{Ad}_{\exp(-th)}(v).
\end{equation}
\begin{rem}\label{rem:Cstar}
By construction, the adjoint quotient $\sigma\colon S\longrightarrow \tfr/W$   is $\C^\ast$-equivariant
with respect to the action \eqref{CstaronS} on $S$ and the action with weights $2d_j$ on 
the coordinate ring $\C[\chi_1,\ldots,\chi_r]$ of $\tfr/W$, where 
$\epsilon_j=d_j-1$ are the exponents of $\gfr$ as before.
In particular, $\sigma$ is \emph{not} $\C^*$-equivariant 
with respect to the standard $\C^*$-action on $\tfr/W$ with weights 
$d_j$. 

As we shall see in greater detail in Theorem \ref{thm:Slosli}, equipped 
with this $\C^\ast$-action the Slodowy slice $S$ is 
$\C^\ast$-equivariantly isomorphic to the corresponding unfolding 
constructed from the Jacobian ring $J_f$, $f\in\C[x,y,z]$, of an 
isolated surface singularity of $\mathrm{ADE}$-type as described before Example \ref{ex:A3-sing}.
For this to be true, the weights of $f$ and hence the $\C^*$-action 
on $J_f$ have to be chosen accordingly  (see also \cite[\S 7.4, Proposition 2]{Slo}). 

Indeed, let $J_f=\C[x,y,z]/(\partial_xf, \partial_yf, \partial_zf)$ as before, and denote by 
$(\mathrm w_x, \mathrm w_y, \mathrm w_z)$ coprime weights of $x, y, z$ such that $f$ is quasi-homogeneous.
If $\gfr$ is of type $A_{2k}$, $k\in\mathbb N$, then one needs to use the weights 
$(\mathrm w_x, \mathrm w_y, \mathrm w_z)$ for $x, y, z$. In all other cases, 
weights $(2\mathrm w_x, 2\mathrm w_y, 2\mathrm w_z)$ must be used for 
$x, y, z$, as we confirm explicitly for our running example in Appendix \ref{app:sh}.
\end{rem}

For the $\Cc$-action, consider the simple adjoint complex Lie 
group $G$ with Lie algebra $\gfr$, and its subgroup 
\begin{equation}\label{innerC}
C(x,y):=\{ g\in G ~|~\mathrm{Ad}_g(x)= x, 
\quad \mathrm{Ad}_g(y) =y\} \quad\subset\quad G,
\end{equation}
which acts on $S$ by the adjoint representation, as one checks by a direct calculation.
Then $\mathrm{Aut}(\Delta_h)$ admits an embedding $\mathrm{Aut}(\Delta_h)\hookrightarrow C(x,y)$ 
	(cf. \cite[Proposition 7.5]{Slo}), so $\Cc$ acts on $S$ via the adjoint action. 
	Moreover, the induced map $\mathrm{Aut}(\Delta_h)\to C(x,y)/C(x,y)^\circ$ 
	to the component group of $C(x,y)$ is bijective by \cite[Proposition 7.5]{Slo}. 
	In summary, $\Cc$ acts on $S$ by inner automorphisms.
	Indeed, for $\gfr\cong\gfr_h^\Cc$ with $\Cc\neq\{\mathrm{id}\}$, $\gfr$ does not have any non-trivial outer
	automorphisms.
	
\begin{rem}\label{rem:CA}
	If $\Delta_h$ is a simply-laced Dynkin diagram and $S_h=x_h+\ker \mathrm{ad}_{y_h}\subset \gfr_h$ is a Slodowy slice over $\tfr_h/W_h$, then $C(x_h,y_h)/C(x_h,y_h)^\circ$ is trivial by \cite[\S 7.5 Lemma 4]{Slo}.
	In other words, one cannot choose an $\mathfrak{sl}_2$-triple $(x_h,y_h,h_h)$ 
	in $\gfr_h$ which is invariant under the natural action of any non-trivial Dynkin graph automorphisms of $\Delta_h$, lifted to the inner automorphisms of $\gfr_h$.
	In fact, one cannot choose $(x_h,y_h,h_h)$ to be invariant 
	under \emph{any} lift 
	of $\Cc$ to $\mbox{Aut}(\gfr_h)$ that yields $\gfr_h^\Cc\cong\gfr$, see Remark \ref{incompatible}.

	If $\Cc$ is non-trivial, then the $\Cc$-action on $S=x+\ker\mathrm{ad}_y$ introduced through \eqref{innerC} must be carefully distinguished from the $\Cc$-action
	on $\gfr_h$ used for folding, even though $S\subset \gfr$ and $\gfr=\gfr_h^\Cc$, i.e. $\gfr\hookrightarrow\gfr_h$. Indeed, the relation to folding is more delicate: 
	let
	\begin{equation}\label{CAdef}
	\CA(x_h,y_h):=\{ \phi\in \mathrm{Aut}(\gfr_h)~|~\phi(x_h)=x_h,\, \phi(y_h)=y_h \},
	\end{equation}
	including all automorphisms of $\gfr_h$ that fix $x_h$ and $y_h$, not just the inner automorphisms as in \eqref{innerC}.
	Then by \cite[\S 7.6, Lemma 2]{Slo} the group of connected components of $\CA(x_h,y_h)$ satisfies
	\begin{equation}\label{eq:isoCA}
	\CA(x_h,y_h)/\CA(x_h,y_h)^\circ \cong \mathrm{Aut}(\Delta_h),
	\end{equation}
	and if the group $\mathrm{Aut}_D(\Delta_h)$ of Dynkin graph automorphisms (see footnote \ref{graphauto}) is non-trivial, then \eqref{eq:isoCA} lifts uniquely 
	to an embedding  
	$\mathrm{Aut}(\Delta_h)\hookrightarrow \CA(x_h,y_h)$.
	In particular, if $\Cc\subset\mathrm{Aut}_D(\Delta_h)$ is a cyclic subgroup, then $\Cc$ acts on $S_h$ by outer automorphisms.
	Note that $\Cc$ does not preserve all fibers of 
	$\sigma_h\colon S_h\longrightarrow \tfr_h/W_h$, but it does preserve the fibers over 
	$(\tfr_h/W_h)^\Cc$, cf. the following Theorem \ref{thm:Slosli} \ref{thm:sloc}.
\end{rem}
For $\Delta$-singularities we now have:

\begin{thm}\label{thm:Slosli}
Let $\tfr\subset \gfr$ be a Cartan subalgebra with Weyl group $W$ and adjoint quotient $\chi\colon\gfr\to \tfr/W$. 
Then the restriction $\sigma:=\chi_{|S}\colon S\longrightarrow \tfr/W$ satisfies the following: 
\begin{enumerate}[label=\alph*)]
\item \cite[$\S 8.3$, $\S 8.4$]{Slo}
The pair $(S_{\bar{0}}, x)$, where $S_{\bar{0}}=\sigma^{-1}(\bar{0})$, together with the induced 
$\Cc$-action is a $\Delta$-singularity. 
\item\label{thm:slob} \cite[\S 8.7]{Slo}
The morphism $\sigma\colon S\longrightarrow \tfr/W$ with the above $(\C^*\times \Cc)$-action 
is a semi-universal $\C^*$-deformation of the $\Delta$-singularity $S_{\bar{0}}$. 
For any $\bar{t}\in \tfr/W$ the singularities of the complex surface $\sigma^{-1}(\bar{t})$ are precisely the finitely many non-regular elements in $\sigma^{-1}(\bar{t})\subset \gfr$, all of which are subregular.
\item\label{thm:sloc} \cite[\S 8.8]{Slo} 
Let $\tfr_h^\Cc\cong \tfr$ and $\gfr_h^\Cc\cong \gfr$ under folding. Moreover, let 
$\sigma_h\colon S_h\longrightarrow \tfr_h/W_h$ be a Slodowy slice with the 
$\Cc$-action defined in Remark \ref{rem:CA} and the $\C^*$-action defined above. 
	Then these two actions commute, and $\sigma_h^{-1}((\tfr_h/W_h)^\Cc)$ is $(\C^*\times \Cc)$-equivariantly 
	isomorphic to $S$ over $(\tfr_h/W_h)^\Cc\cong \tfr/W$.
\end{enumerate}
\end{thm}
For the respective bases of semi-universal $\C^\ast$-deformations, the above
theorem implies the natural isomorphisms $J_f\cong\tfr_h/W_h$ 
and $J_f^\Cc\cong(\tfr_h/W_h)^\Cc\cong\tfr/W$ as $\C^*$-spaces.
\begin{rem}\label{incompatible}
	We emphasize that an isomorphism 
	$\sigma_h^{-1}((\tfr_h/W_h)^\Cc)\cong S$ as in 
	Theorem \ref{thm:Slosli} \ref{thm:sloc} cannot be induced by a Lie algebra homomorphism 
	$\gfr\longrightarrow \gfr_h$ if $\Cc\neq\{\id\}$, see \cite[\S 8.7 Remark 3)]{Slo}. 
	As was already mentioned in Remark \ref{rem:CA}, this is related to the fact that 
	there is no $\mathfrak{sl}_2$-triple $(x_h,\, y_h,\, h_h)$ in $\gfr_h$ which 
	is invariant under $\Cc$ and with  subregular nilpotent $x_h,\, y_h$. Indeed, such $x_h,\, y_h$ would yield Slodowy slices $S_h=x_h+\ker \mathrm{ad}_{y_h}\subset \gfr_h$ and 
	$S=x_h+\ker(\mathrm{ad}_{y_h})_{\mid\gfr_h^\Cc}\subset \gfr$ and a $\Cc$-equivariant
	map $S\hookrightarrow S_h$ induced by a Lie algebra homomorphism $\gfr\hookrightarrow\gfr_h$,
	contradicting the above-mentioned Remark 3) of \cite[\S8.7]{Slo}. In fact, as we have seen
	in Remark \ref{rem:CA}, $\Cc$	acts by inner automorphisms on $S$ but 
	by outer automorphisms on $S_h$.
	In this sense, folding of $\mathrm{ADE}$-singularities is not compatible 
	with folding of $\mathrm{ADE}$-Lie algebras. 
\end{rem}	
	For any irreducible Dynkin diagram $\Delta$ we will always 
	work with a Slowody slice $S$ inside the simple Lie algebra $\gfr(\Delta)$.
	In particular, the $\Cc$-action on $S$ is always defined by lifting $\Cc$ to $G$ 
	so that it acts by the adjoint action.

\begin{ex}\label{runslices}
To construct a Slodowy slice $S$ for $\mathfrak{g}(\mathrm{C}_2)=\mathfrak{sp}_4(\C)$, we take the 
$\mathfrak{sl}_2$-triple $(x,y,h)$ with 
$$
x=\left( \begin{matrix} 0&0&1&0\\ 0&0&0&1\\  0&0&0&0\\  0&0&0&0\end{matrix}\right),\quad
y=x^T=\left( \begin{matrix} 0&0&0&0\\ 0&0&0&0\\  1&0&0&0\\  0&1&0&0\end{matrix}\right),\quad
h=[x,y] = \mbox{diag}\left(1,1,-1,-1\right).
$$
Note that there does not exist any nilpotent subregular  $x\in \gfr_h=\mathfrak{sl}_4(\C)$
with $x\in\mathfrak{sp}_4(\C)$, in accord 
with the above Remark \ref{incompatible}. 
The element $x$ chosen above is subregular
in $\mathfrak{sp}_4(\C)$ but not in $\mathfrak{sl}_4(\C)$.
For the resulting Slodowy slice, by a direct comptutation we obtain from the definition of $S$ that
\begin{eqnarray*}
	S = x + \ker \mathrm{ad}_y
	&=& \left\{ \left.  s(v_1^-,\,v_2^-,\,v_1^+,\,v_2^+)
	\right|
	v_1^\pm,\, v_2^\pm\in \C \right\}\\ 
	&&  \mbox{ with } 
	s(v_1^-,\,v_2^-,\,v_1^+,\,v_2^+) :=
	\left( \begin{matrix} 0&v_1^-&1&0\\ -v_1^-&0&0&1\\ v_2^++v_2^-&v_1^+&0&v_1^-\\v_1^+&v_2^+-v_2^-&-v_1^-&0  \end{matrix} \right).
\end{eqnarray*}
The $\Cc$-action on $S$ is given by conjugation with an appropriate 
element $M$ in the group $C(x,y)$ defined in \eqref{innerC}. We find
$$
C(x,y)=\left\{ \left.\left( \begin{matrix} K&0\\0&K \end{matrix}\right) \right|
K\in\mbox{Mat}_{2\times2}(\C),\, K\cdot K^T=\mathbf{1}\right\},
$$
so $C(x,y)$ has two connected components which are distinguished by $\det(K)\in \{ \pm 1\}$.
We choose
$$
M=\left( \begin{matrix} \mathbf{Q}&0\\0&\mathbf{Q}\end{matrix}\right)
\qquad \mbox{ where }\qquad \mathbf{Q}=\left( \begin{matrix} 0&1\\1&0\end{matrix}\right),
$$
cf.~Example \ref{rungroupfold}, as a representative of the generator of $C(x,y)/C(x,y)^\circ\cong \Z/2\Z$.
The induced action on $S$ in terms of the $v_k^\pm$ is then $v_k^+\mapsto v_k^+,\, v_k^-\mapsto-v_k^-$.
Continuing to use the same notations as in Examples \ref{runLiealgebras} and \ref{runchis}, we find
\begin{eqnarray}\label{inhomslice}
\hspace*{-1em} \xi\circ\chi\colon 
S &\longrightarrow& \xi\left( \tfr/W\right)=\C^2
\nonumber\\[10pt]
\hspace*{0.5em}
s(v_1^-,\,v_2^-,\,v_1^+,\,v_2^+)
&\longmapsto&
( 2(v_1^-)^2-2v_2^+,\quad
(v_1^-)^4+2(v_1^-)^2v_2^+  +(v_2^+)^2 - (v_2^-)^2 -(v_1^+)^2).
\hphantom{\xi_h\circ\chi_h}\nonumber
\end{eqnarray}
The $\C^\ast$-action of $\lambda\in\C^\ast$, $\lambda=e^t$, on $S$ is obtained by conjugation with  the matrix $\exp(-th)=\mbox{diag}(\lambda^{-1},\lambda^{-1},\lambda,\lambda)$, followed by multiplication by $\lambda^2$. 
For the $v_k^\pm$ we thereby read off:
$$
\lambda\cdot(v_1^-,v_2^-,v_1^+,v_2^+)= (\lambda^2 v_1^-,\lambda^4 v_2^-,\lambda^4 v_1^+, \lambda^4 v_2^+).
$$
In Appendix \ref{app:sh} we construct a Slodowy slice $S_h$ in 
$\gfr_h=\mathfrak{sl}_4(\C)$ and construct an explicit isomorphism to 
$S$ over $(\tfr_h/W_h)^\Cc\cong \tfr/W$ in accordance with 
Theorem \ref{thm:Slosli} \ref{thm:sloc}.
\end{ex}
\section{Folding of Hitchin systems}\label{s:foldinghitchin}
From an algebraic-geometric point of view, Hitchin systems are a 
global analogue of the adjoint quotient as we explain in the next subsection.
The aim of this section is to introduce the notion of folding of Hitchin 
systems and to globalize the commutative diagram \eqref{eq:relationadjquotients} 
to Hitchin systems over the locus of smooth fibers, see Theorem \ref{thm:hitchinfibers}.

To make this precise, we  fix a Riemann surface $\cu$ of genus 
\fontdimen16\textfont2=4pt
$g_\cu\geq 2$. 
\fontdimen16\textfont2=2.5pt
Given any reductive complex algebraic group $H$, by $\mh_0(\cu,H)$ we denote 
the \emph{neutral component} of the moduli space of (S-equivalence classes of) semistable principal 
$H$-Higgs bundles or, equivalently, of polystable $H$-Higgs bundles of degree zero.
In fact, since we restrict to the neutral component throughout this work\footnote{We do so because so far, only this component has been related to Calabi--Yau integrable systems. 
Other components hopefully also arise geometrically, perhaps by G-flux.}, 
to unclutter the notation, we drop the index $0$ altogether and denote by 
$\mh(\cu,H):=\mh_0(\cu,H)$
the neutral component of the moduli space, hoping that this will not lead to confusion.
Then $\mh(\cu,H)$ is a quasi-projective variety \cite{Hit1,Nitsure,SimpsonModuliI,SimpsonModuliII} whose complex 
structure is induced by that of $\Sigma$ and $H$.
It contains the cotangent bundle $T^*\mathcal{N}(\cu,H)^{sm}$ of the smooth locus 
of the moduli space $\mathcal{N}(\cu,H)$ of semistable/polystable $H$-bundles as an open and dense subset.
The canonical symplectic structure on $T^*\mathcal{N}(\cu,H)^{sm}$ extends to a holomorphic symplectic structure $\omega_{\mh}$ on the smooth locus $\mh(\cu,H)^{sm}$ of $\mh(\cu,H)$ by \cite[Theorem 4.3]{BR}.

We next recall the structure of an algebraic integrable system on the 
holomorphic symplectic manifold $(\mh(\cu,H)^{sm},\omega_{\mh})$.

\subsection{Hitchin systems}
As above, let $H$ denote a reductive complex algebraic group.
Its associated Lie algebra is $\hfr$, and $\tfr\subset \hfr$ denotes 
a Cartan subalgebra with Weyl group $W$.
Then the variety $\mh(\cu,H)$ possesses the structure of an algebraic integrable system
\cite[\ldots]{Hit1, Hit2, BNR, BR, Faltings,DG}.
Recall that an algebraic integrable system is a flat and surjective holomorphic 
map $\pi\colon (M,\omega)\longrightarrow B$ from a holomorphic symplectic manifold 
$(M,\omega)$ to a complex manifold $B$ with the following properties:
\begin{itemize}
	\item 
	the smooth loci of the fibers of $\pi$ are Lagrangian in $(M,\omega)$,
	\item 
	there exists a Zariski-dense subset $B^\circ\subset B$ such that the restriction 
	of $\pi$ to $\pi^{-1}(B^\circ)$ has connected and compact fibers and admits a relative 
	polarization, that is, there exists a line bundle $\mathcal L\to M$ whose restriction to $\pi^{-1}(b)$ is ample for every $b\in B^\circ$.
\end{itemize} 
The holomorphic version of the Arnold--Liouville theorem implies that the fibers 
over $B^\circ$ are torsors for abelian varieties. 
For the construction of an algebraic integrable system on $\mh(\cu,H)$,
recall from Section \ref{sec:foldingCartan} that $\tfr/W$ carries a natural $\C^*$-action. 
Considering the canonical bundle $K_\cu$ of $\cu$ as a $\C^*$-bundle, we define the total space
\begin{equation*}
\Ub:=\mathrm{tot}(K_\cu\times_{\C^*} \tfr/W)
\end{equation*}
of the bundle $K_\cu\times_{\C^*} \tfr/W$.
Its space of holomorphic sections is denoted by 
$\Bb(\cu,H):=H^0(\cu,\Ub)$.
The adjoint quotient $\chi\colon \hfr\to\tfr/W$ induces the Hitchin map
\begin{equation*}
\bc\colon\mh(\cu,H)\longrightarrow\Bb(\cu,H).
\end{equation*}
On the smooth locus equipped with the holomorphic symplectic structure, this is an algebraic 
integrable system, called the $H$-Hitchin system. 
By the above, $\bc$ may be considered as global version of the adjoint quotient $\chi$.
Many properties of $\chi$ carry over to $\bc$.
For example, the $\C^\ast$-equivariance of $\chi$ is inherited by $\bc$. 
As mentioned in Section \ref{sec:foldingCartan}, one may choose homogeneous generators of $\C[\tfr]^{W}$ of degrees $d_1,\,\ldots,\, d_r$, where $d_1-1,\,\ldots,\, d_r-1$ are the exponents of $\hfr$.
One then obtains a convenient decomposition of $\Bb:=\Bb(\cu,H)$ as
\begin{equation}\label{Bdecomposes}
\Bb\cong\bigoplus_{j=1}^{r} H^0(\cu, K_\cu^{d_j}) .
\end{equation}
Moreover, $\bc$ admits (Lagrangian) sections \cite[\S 5]{Hit3},  so-called Hitchin sections.

In accord with our conventions in Section \ref{s:folding}, we write $H=G$ if $H=(G_h^\Cc)^\circ$ arises from folding a simple \emph{adjoint} complex Lie group $G_h$ with Dynkin diagram of ADE type as in Proposition \ref{foldedgroup}.
If we choose $H=G_h$, then we decorate all our notations with the subscript $h$.
Now the inclusion $\tfr/W\cong (\tfr_h/W_h)^\Cc \hookrightarrow \tfr_h/W_h$ 
obtained in Corollary \ref{cor:tfr} induces the isomorphism 
\begin{equation}\label{BisBhC}
\iota\colon\Bb \stackrel{\cong}{\longrightarrow} \Bb_h^\Cc 
\qquad\mbox{ with } \qquad
\Bb=\Bb(\cu,G),\quad \Bb_h=\Bb(\cu,G_h), \qquad \Bb_h^\Cc \subset \Bb_h.
\end{equation}
The Hitchin sections mentioned above  are the analogues of the Kostant sections
of $\chi$ and $\chi_h$, respectively. In Theorem \ref{thm:hitchinfibers} below,
we further extend this analogy and establish a global version of the commutative diagram 
\eqref{eq:relationadjquotients}, at least over the locus of smooth fibers of 
$\bc$ and $\bc_h$, see equation \eqref{eq:isointsys}.
\begin{ex}\label{runHiggs}
In our running example with $G_h=PSL_4(\C)$ and $G=PSp_4(\C)$ (see Example \ref{rungroupfold}), it is convenient to first introduce $SL_4(\C)$- and $Sp_4(\C)$-Higgs bundles. 
A principal $SL_4(\C)$-Higgs bundle is equivalent to a Higgs vector bundle $(E,\varphi)$ consisting of a holomorphic vector bundle $E$ over $\cu$ of rank $4$ with trivial determinant bundle $\Lambda^4 E$ and a trace-free Higgs field $\varphi\in H^0(\cu, K_\cu\otimes \mathrm{End}_0(E))$.
	
Analogously, a principal $Sp_4(\C)$-Higgs bundle is equivalent to a pair $((F,\langle\cdot , \cdot\rangle), \psi)$ of a holomorphic vector bundle $F$ of rank $4$ with symplectic pairing $\langle\cdot , \cdot\rangle$ and a holomorphic section $\psi\in H^0(\cu, K_\cu\otimes \mathrm{End}(F))$ which is skew-symmetric with respect to $\langle \cdot, \cdot \rangle$.
In particular, $\psi$ is trace-free. 

Then the moduli space of semistable principal $PSL_4(\C)$-Higgs bundles is described in \cite[\S 2]{HauselThaddeus}; the result generalizes to any 
classical reductive complex Lie group
of adjoint type and reads
\begin{equation}\label{eq:MHex}
\begin{aligned}
\mathcal{M}(\cu,PSL_4(\C))&\cong   \mathcal{M}(\cu,SL_4(\C))/J(\cu)[4],\\ 
\mathcal{M}(\cu,PSp_4(\C))& \cong \mathcal{M}(\cu,Sp_4(\C))/J(\cu)[4].
\end{aligned}
\end{equation}
Recall here that we are working with the neutral 
components $\mh(\cu,H)=\mh_0(\cu,H)$ of the moduli 
spaces, throughout.
In \eqref{eq:MHex} the elements of $J(\cu)[4]$ of order $4$ in the Jacobian $J(\cu)$ of $\cu$ act on the respective moduli 
spaces by taking the tensor product.

Choosing the same generators of $\C[\tfr_h]^{W_h}$ and $\C[\tfr]^W$ as in Example \ref{runLiealgebras}, we obtain the isomorphisms
\begin{equation}\label{eq:runBBh}
\Bb_h\cong H^0(\cu,K_\cu^2)\oplus H^0(\cu, K_\cu^3)\oplus H^0(\cu, K_\cu^4),\quad \Bb\cong H^0(\cu,K_\cu^2)\oplus H^0(\cu, K_\cu^4).
\end{equation}
Hence we write elements of $\Bb_h$ and $\Bb$ as $b^h=(b_2^h,b_3^h,b_4^h)$ and $b=(b_2,b_4)$ respectively.
Then the Hitchin maps are given by
\begin{equation*}
\bm{\chi}_h([(E,\varphi)])=\left(\mathrm{tr}(\Lambda^2 \varphi),\mathrm{tr} (\Lambda^3 \varphi), \mathrm{tr}(\Lambda^4 \varphi) \right),\quad \bm{\chi}_h([(F,\psi)])=\left(\mathrm{tr}(\Lambda^2 \psi),\mathrm{tr}( \Lambda^4 \psi)\right).
\end{equation*}
\end{ex}

\subsection{Generic Hitchin fibers}\label{sec:genericHitch}
To give the isomorphism classes of the generic Hitchin fibers, 
i.e.~the generic fibers of $\bc$, we introduce the total space
\begin{equation*}
	\widetilde\Ub:=\mathrm{tot}(K_\cu\times_{\C^*} \tfr)
\end{equation*}
of the bundle associated to the standard $\C^*$-action on $\tfr$.
The quotient map $q\colon \tfr\to \tfr/W$ globalizes to give
\begin{equation*}
\bq\colon \widetilde{\Ub}\to \Ub.
\end{equation*}
Then the cameral curve $\tcu_{b}$ associated to $b\in \Bb$ is defined by the fiber product
\begin{equation*}
\begin{tikzcd}
\tcu_{b} \ar[r] \ar[d, "p_b"'] & \widetilde{\Ub} \ar[d, "\bq"] \\
\cu \ar[r,"b"] & \Ub
\end{tikzcd}
\end{equation*}
If the fiber $\bc^{-1}(b)$, $b\in \Bb$, is 
smooth, then it is a torsor for the abelian variety $P_b$.
If $\tcu_b$ is smooth, then $P_b$ is explicitly determined by $\tcu_b$, 
see \cite[\S 4]{DG} for the most general treatment.
For our purposes we only need to consider the case where $H$ is a simple complex 
Lie group of \emph{adjoint type}, which we assume from now on.

Let $\bLambda$ be the cocharacter lattice of $H$ and $b\in \Bb$ such that $\tcu_b$ is smooth.
Then the cocharacter lattice of $P_b$ is given by
\begin{equation}\label{eq:cocharP}
\mathrm{cochar}(P_b)=H^1(\cu,(p_{b}{}_\ast\bLambda)^W),
\end{equation}
see \cite[\S 3]{DP} and \cite[Corollary 3.3.2 and Proposition 3.4.1]{Beck}. 
The complex structure is determined by the universal coverings
\begin{equation}\label{eq:universalcoveringP}
H^1(\tcu_b,\tfr)^W 
\end{equation}
of the abelian varieties $P_b$.
In summary, \eqref{eq:cocharP} completely 
determines the isomorphism class of $P_b$ and hence of $\bc^{-1}(b)$.

\subsection{Folding of cameral curves}\label{ss:foldingcameral}
As a first step towards folding Hitchin systems, in this section we show
how to fold the cameral curves that we introduced in Section \ref{sec:genericHitch}.
In the following, $G=(G_h^\Cc)^\circ$ arises from folding a simple complex Lie group $G_h$ of adjoint type.
We further use the notation set up in \eqref{BisBhC}.

\begin{prop}\label{p:foldingcameral}
Let $b\in \Bb$ be transversal to the discriminant locus $\mathrm{discr}(\bm{q})$
of $\bm{q}$.
Then $\widetilde{\cu}_b$ and $\widetilde{\cu}_{h,\iota(b)}$ are
\emph{smooth} $W$- and $W_h$-cameral curves respectively. 
They are related via
\begin{equation}\label{eq:cams}
\widetilde{\cu}_{h,\iota(b)}\cong(\widetilde{\cu}_b\times W_h)/W. 
\end{equation}
Here $W\subset W_h$ acts diagonally on $\widetilde{\cu}_b\times W_h$. 
Moreover, let $\alpha\in R_h$, $\beta= \alpha^O$, so 
$\beta\in R_h^\Cc=R^\vee$, cf.~equation \eqref{eq:alphaO}. Then the local 
monodromies around the branch points of 
$p_{\iota(b)}^h\colon\widetilde{\cu}_{h,\iota(b)}\longrightarrow \cu$ are given by the restrictions of 
$s_\beta=\prod\limits_{\gamma \in O(\alpha)} s_\gamma$, cf.~equation \eqref{eq:weylfolding}.
\end{prop}
\begin{rem}
	If $b\in \Bb$ is transversal to $\mathrm{discr}(\bq)$, then $\iota(b)\in \Bb_h$ cannot be transversal to $\mathrm{discr}(\bq_h)$ for $\Cc\neq\{\id\}$. 
	Indeed, by the compactness of $\cu$, it is easy to see that such 
	$b$ has to (transversally) intersect each hyperplane 
	$\Ub^{\alpha^O}=K_\cu\times_{\C^*}\tfr^{\alpha^O}/W\subset \Ub_h$, 
	$\alpha^O\in R^\vee$, with $\tfr^{\alpha^O}$ as in \eqref{eq:weylfolding}.
	Let $\alpha\in R$ with $|O(\alpha)|>1$ and let $x\in \cu$ be a point such that $b(x)\in \Ub^{\alpha^O}$.
	For any $\tilde{u}\in \bq^{-1}(b(x))$, by \eqref{eq:tfrO} we have $\dim\mbox{ker} (d\bq_{h,\tilde{u}}) = |O(\alpha)|>1$ by assumption, so
	$$
	\mathrm{rk}(d\bq_{h,\tilde{u}})+1 < \dim(\widetilde{\Ub}_h) = \dim(\Ub_h).
	$$
	Hence $\iota(b)$ cannot intersect $\mathrm{discr}(\bq_h)$ transversally at $x$.
	In particular, we cannot conclude from the general theory that $\tcu_{h,\iota(b)}$ is smooth.
\end{rem}

\begin{proof}[Proof of Proposition \ref{p:foldingcameral}] 
We first consider the local situation. 
Let $D\subset \Sigma$ be a sufficiently small open subset and 
$b\colon D\longrightarrow \tfr/W\subset \tfr_h/W_h$ be a holomorphic map which is transversal to $\mathrm{discr}(q)$.
Further note that $q_h^{-1}(\tfr/W)=W_h(\tfr)$ fits into the following commutative diagram
\begin{equation*}
\begin{tikzcd}
\tfr \ar[r, hookrightarrow] \ar[dr, "q"'] & W_h(\tfr) \ar[d, "q_h"] \\
& \tfr/W.
\end{tikzcd}
\end{equation*}
Since $b$ is transversal to $\mathrm{discr}(q)$, the commutativity of the above diagram implies that 
$b$ is transversal to $\mathrm{discr}(q_{h|W_h(\tfr)})$ as well. 
Hence the cameral curves are locally given by
the smooth coverings $\widetilde{D}_b=D\times_{\tfr/W} \tfr $ 
and $\widetilde{D}_{h,\iota(b)}=D\times_{\tfr/W} W_h(\tfr)$ 
of $D$ and are therefore smooth. This proves
the first statement of the proposition. 

To prove \eqref{eq:cams} we next observe
\begin{equation}\label{eq:tth}
W_h(\tfr^\circ)\cong (\tfr^\circ\times W_h)/W,
\end{equation}
where $\tfr^{\circ}$ denotes the set of semisimple regular elements in 
$\tfr$ and $W$ acts by $w\cdot (t,w_h)=(w\cdot t, w_h w^{-1})$.
Indeed, by Lemma \ref{regact} the map 
\begin{equation}
(\tfr^\circ\times W_h)/W\to W_h(\tfr^\circ),\quad [(t,w)]\mapsto w\cdot t
\end{equation}
is a well-defined isomorphism of $W_h$-coverings of $\tfr^\circ/W$. 

Next observe that the branch locus of $p_{\iota(b)}^h\colon\tcu_{h,\iota(b)}\longrightarrow \cu$ 
in $\cu$ coincides with that of 
$p_b\colon\tcu_b\longrightarrow \cu$. 
We denote by $\cu^\circ$ its complement and by $\tcu_{h,\iota(b)}^\circ$ and $\tcu_{b}^\circ$ 
the complements of the ramifications divisors of $p_{\iota(b)}^h$ and $p_b$ respectively. 
The latter are the pullbacks of $K_\cu \times_{\C^\ast} W_h(\tfr^\circ)$ 
and $K_\cu \times_{\C^\ast} \tfr^\circ$ along $\iota(b)$ and $b$ respectively. 
Hence (\ref{eq:tth}) implies
\begin{equation*}
\tcu^\circ_{h,\iota(b)}\cong (\tcu^\circ_b\times W_h)/W
\end{equation*}
as $W_h$-coverings over $\cu^\circ$. 
Here $W$ acts diagonally as before. 
Since the local monodromies around the points of $\cu-\cu^\circ$ coincide and $\tcu_{h,\iota(b)}$ as well as $\tcu_b$ are smooth, the previous isomorphism extends to the isomorphism (\ref{eq:cams}) of simply branched $W_h$-Galois coverings of $\cu$.  

The  claim about the local monodromies now follows 
from the prescription for folding $W_h$ to $W$, 
see Proposition \ref{p:foldingweyl}, in particular formula \eqref{eq:weylfolding}. 
\end{proof}

In the setup of Proposition \ref{p:foldingcameral}, $\widetilde{\cu}_{h,\iota(b)}$ decomposes 
non-canonically into $[W_h:W]$ copies of $\widetilde{\cu}_b$: 
fix a representative $w_j\in W_h$ for each equivalence class in $W_h/W$, where we choose the canonical representative $1\in W$ for the class $W$.
Then we obtain the inclusions
\begin{equation}\label{eq:i_wh}
i_{w_{j}}\colon\widetilde{\cu}_b\hookrightarrow \widetilde{\cu}_{h,\iota(b)},\quad 
x\mapsto [(x,w_{j})]
\end{equation}
under the isomorphism \eqref{eq:cams}.
In particular, with $i:=i_{1}$, we have $i_{w_{j}}=w_j\circ i$. 
By the universal property of the coproduct, these inclusions induce the morphism
\begin{equation*}
\coprod_{[w_{j}]\in W_h/W} \widetilde{\cu}_b\longrightarrow \widetilde{\cu}_{h,\iota(b)}
\end{equation*}
which is a $W_h$-equivariant isomorphism. 

\begin{ex}
The heart of the previous proof is \eqref{eq:tth} which is equivalent to the non-canonical decomposition 
	\begin{equation}\label{eq:decompWh}
	W_h(\tfr^\circ)\cong \coprod_{[w_j]\in W_h/W} \tfr^\circ.
	\end{equation}
In our running example, we have $W_h=S_4$ and $W=\Z/4\Z\rtimes \Z/2\Z$ so that $[W_h:W]=3$. 
In fact, with notations as in Example \ref{runLiealgebras}, $w\in S_4$ acts on the
elements of $\tfr_h$ by 
$$
\mbox{diag}\left( x_1,\, x_2,\, x_3,\, x_4\right)
\longmapsto \mbox{diag}\left( x_{w(1)},\, x_{w(2)},\, x_{w(3)},\, x_{w(4)}\right),
$$
and $W\subset W_h$ is generated by the reflections in $(\beta_1^\vee)^\perp,\, 
(\beta_2^\vee)^\perp$ which act as $w=(1,\,4)(2,\,3)$ and  $w=(2,\,4)$.
By Example \ref{runLiealgebras} we further know
\begin{equation*}
\tfr^\circ=\{ \mathrm{diag}(u,\,v,\,-u,\,-v)~|~u,\, v\in\C^\ast,\, u\neq \pm v \}.
\end{equation*}
It follows that $W_h(\tfr^\circ)\subset \tfr_h$ decomposes into three connected
components, according to \eqref{eq:decompWh},
\begin{gather*}
\{ \mathrm{diag}(u,\,v,\,-u,\,-v)~|~u,\, v\in\C^\ast,\, u\neq \pm v \} 
\coprod \{ \mathrm{diag}(u,\,v,\,-v,\,-u)~|~u,\, v\in\C^\ast,\, u\neq \pm v \}\\
\coprod \{ \mathrm{diag}(u,\,-u,\,v,\,-v)~|~u,\, v\in\C^\ast,\, u\neq \pm v \}.
\end{gather*}
This example also shows that the  decomposition 
\eqref{eq:decompWh} ceases to be disjoint when $\tfr^\circ$ is replaced by a subset of $\tfr$ which contains non-regular elements.
\end{ex}

\subsection{Folding of Hitchin systems}\label{ss:foldinghit}
For any homomorphism $\sigma\colon H\longrightarrow H^\prime$ of reductive complex 
algebraic groups and any $H$-Higgs bundle $(F,\theta)$, we denote by $(\sigma(F),\sigma_*(\theta))$ 
the associated $H^\prime$-Higgs bundle obtained by the extension of the structure group.
In particular, the automorphism group $\mathrm{Aut}(G_h)\cong\mathrm{Aut}(\gfr_h)$ 
acts naturally on $\mh(\cu,G_h)$ via
\begin{equation}
[(E,\varphi)]\mapsto [(\sigma(E),\sigma_*(\varphi))].
\end{equation}
Inner automorphisms act trivially so that the $\mathrm{Aut}(G_h)$-action descends to an 
$\mathrm{Aut}(\Delta_h)$-action, in particular to an 
action by $\Cc=\langle a \rangle \subset \mathrm{Aut}(\Delta_h)$.
\begin{prop}\label{inclusion}
	The inclusion $G\hookrightarrow G_h$ of simple adjoint complex Lie groups induces an
	\emph{injective} holomorphic map
	\begin{equation*}
	\iota_{\mh}\colon\mh(\cu,G)\longrightarrow \mh(\cu,G_h)^\Cc
	\end{equation*}
	over $\iota(\Bb)\subset \Bb_h$. 
\end{prop}
As a preparation for our proof, we need the following
lemma which seems to be known to the experts (see 
e.g.\cite[\S 5]{GPR}), but we include a proof for completeness.
\begin{lem}\label{lem:normalizer}
	Let $i_H\colon H\hookrightarrow G_h$ be the inclusion of a reductive subgroup 
	and $i_N\colon H \hookrightarrow N_{G_h}(H)$ the inclusion into its normalizer $N=N_{G_h}(H)$.
	Further let $(E_j,\varphi_j)$, $j \in\{1,2\}$, be two $H$-Higgs bundles such that 
	$(i_H(E_j),i_{H,*}(\varphi_j))$ are isomorphic $G_h$-Higgs bundles. 
	Then $(i_N(E_j),i_{N,*}(\varphi_j))$ are isomorphic $N$-Higgs bundles. 
\end{lem}
\begin{proof}
Let $f\colon (i_H(E_1),i_{H,*}(\varphi_1))\to (i_H(E_2),i_{H,*}(\varphi_2))$ be an isomorphism of $G_h$-Higgs bundles.
It is proven in \cite[Lemma 5.11]{GPPNR} that $f$ restricts to an isomorphism $f_N\colon i_N(E_1)\to i_N(E_2)$ of $N$-bundles.
Since $i_{H,*}(\varphi_j)$ is (locally) valued in $\mathrm{Lie}(H)\subset \mathrm{Lie}(N)\subset \gfr_h$, $f_N\colon (i_N(E_1),i_{N,*}(\varphi_1))\to (i_N(E_2),i_{N,*}(\varphi_2))$ is an isomorphism of $N$-Higgs bundles. 
\end{proof}

\begin{proof}[Proof of Proposition \ref{inclusion}] 
Let $\jmath\colon G\longrightarrow G_h$ be the group monomorphism obtained by folding by $\Cc=\langle a \rangle$. 
We claim that $(\jmath(E),\jmath_*(\varphi))$ is polystable as a $G_h$-Higgs bundle if $(E,\varphi)$ is polystable as a $G$-Higgs bundle. 
To this end, recall that an $H$-Higgs bundle $(F,\theta)$ is polystable iff the 
$GL(\mathfrak{h})$-Higgs bundle $(\mathrm{Ad}(F),\mathrm{Ad}_*(\theta))$ 
associated to $(F,\theta)$ via the adjoint representation 
$\mathrm{Ad}\colon H\longrightarrow GL(\mathfrak{h})$ is polystable \cite[Lemma 4.7]{AnchoucheBiswas}. 
Hence we may assume $G \subset G_h \subset GL(n)$.
The Higgs vector bundle $(E,\varphi)$ is polystable iff there 
exists a Hermitian metric on $E$ 
whose Chern connection $\nabla$ (taken with respect to the given holomorphic structure on $E$) satisfies
\begin{equation}\label{eq:NAH}
F^\nabla
+[\varphi,\varphi^\dagger  ]=0, \quad 
\bar{\partial}^\nabla \varphi=0,
\end{equation}
see \cite[Theorem 1]{SimpsonHiggs}.
Since \eqref{eq:NAH} is still satisfied after passing from the polystable Higgs vector 
bundle $(E,\varphi)$ to the Higgs vector bundle $(\jmath(E),\jmath_*(\varphi))$, the latter is polystable as well. 
Therefore 
\begin{equation}
\iota_{\mh}\colon \mh(\cu, G)\longrightarrow \mh(\cu,G_h),\quad [(E,\varphi)]\mapsto [(E_h,\varphi_h)]
\end{equation}
is well-defined.
The equality $a\circ \jmath=\jmath$ implies that the image of $\iota_{\mh}$ lies in $\mh(\cu,G_h)^\Cc$.
Since $\bc_h$ is $\Cc$-equivariant, $\iota_{\mh}$ is defined over $\iota(\Bb)\subset \Bb_h$. 

We next show the injectivity of $\iota_{\mh}$.
In \cite[Proposition 2.20]{GPR} it is proven\footnote{This proposition is stated for $\mathrm{ord}(a)=2$ but generalizes to arbitrary finite order.} that
	\begin{equation*}
	N_{G_h}(G_h^{\Cc})
	=\{ g\in G_h~|~a(g)=c(g)g\quad \mbox{for some }c(g)\in Z(G_h)  \}.
	\end{equation*}
Since $G\cong (G_h^\Cc)^\circ$ by Proposition \ref{foldedgroup} and $Z(G_h)=\{\id\}$, we conclude $G=N_{G_h}(G)$.
Hence $\iota_{\mh}$ is injective by Lemma \ref{lem:normalizer}.
\end{proof}
\begin{ex}
Following Example \ref{runHiggs}, we use \eqref{eq:MHex} to
represent isomorphism classes of Higgs vector bundles that correspond to
semistable principal $PSL_4(\C)$- and $PSp_4(\C)$-Higgs 
bundles, respectively. 
Then the inclusion of Proposition \ref{inclusion} is given by 
	\begin{equation*}
	\iota_\mh\colon \mh(\cu, PSp_4(\C))\to \mh(\cu,PSL_4(\C)),\quad [(F,\langle \cdot, \cdot \rangle,\psi)] \mapsto [(F,\psi)].
	\end{equation*}
Note that $(F,\psi)$ is indeed an $SL_4(\C)$-Higgs vector bundle because it is an $Sp_4(\C)$-Higgs vector bundle, so in particular $\langle \cdot, \cdot \rangle$ induces a trivialization of $\Lambda^4 F$ and $\psi$ is a trace-free Higgs field.
\end{ex}

\begin{rem}\label{rem:GPR1}
	In \cite[Theorem 6.3]{GPR} 
	it is shown that if $\Cc\neq\{\id\}$, then $\mathcal{M}(\cu,G)\subsetneq \mathcal{M}(\cu,G_h)^\Cc$.
	For example if $G_h=PSL_{2n}(\C)$, then 
	\begin{equation*}
	\mh(\cu,PSO_{2n}(\C))\cup \mathcal{M}(\cu, PSp_{2n}(\C))\hookrightarrow \mh(\cu,PSL_{2n}(\C))^\Cc,
	\end{equation*}
	see \cite[Section 9.2]{GPR}.
	Observe that $G=PSp_{2n}(\C)$ is the group obtained by folding by the 
	involution $a(A)=J(A^T)^{-1}J^{-1}$, where $J$ is the 
	standard symplectic matrix as in Example \ref{rungroupfold}.
	On the other hand, $PSO_{2n}(\C)$ arises from folding by the involution $\sigma(A)=(A^T)^{-1}$. 
	Note that both $a$ and $\sigma$ induce the same outer automorphism on the Dynkin diagram.
	\end{rem}
	In the remainder of this section, we show that 
	$\mh(\cu,G_h)^\Cc- \iota_\mh(\mh(\cu,G))$ is empty over a 
	suitable open and dense subset of $\iota(\Bb)$. 

\begin{prop}\label{p:isosheaves}
Choose an arbitrary $b\in \Bb$ which is transversal to $\mathrm{discr}(\bm{q})$.
	 Then the sheaves $((p_{\iota(b)}^h)_\ast\bLambda_h)^{W_h}$ 
	 and $({p_{b}}_\ast\bLambda_h)^W$ 
	 are naturally isomorphic to each other. 
\end{prop}
\begin{proof}
	By \eqref{eq:i_wh} we have $p_b=p_{\iota(b)}^h\circ i$ and
	\begin{gather*}
	(p_{\iota(b)}^h)_\ast(\bLambda_h)\cong  \bigoplus_{[w_j]\in W_h/W}{p_b}_\ast(\bLambda_h), \quad
	f\mapsto (f_j)_{[w_j]\in W_h/W},
	\end{gather*}
	Here the sections $f_j$ are uniquely determined by
	\begin{equation*}
	f_j \circ i= (w_j^{-1})^\ast (f_{| i_{w_j}(\widetilde\cu_b)}).
	\end{equation*}
	Restricting this isomorphism to $W_h$-equivariant sections $f$, one obtains $f_j\circ i =f_{i(\tcu_b)}\in ({p_b}_\ast(\bLambda_h))^W$ by the $W_h$-equivariance and $i_{w_j}=w_j\circ i$.
	Hence $f\mapsto f_1$ 
	yields the desired isomorphism $((p_{\iota(b)}^h)_\ast\bLambda_h)^{W_h}\cong 
	({p_b}_\ast(\bLambda_h))^W$.
\end{proof}

\begin{thm}\label{thm:hitchinfibers}
Let $\Bbo\subset \Bb$ denote the open set of sections that are transversal to $\mathrm{discr}(\bq)$. 
Then for any $b\in\Bbo$ the Hitchin fibers $P=P_b$ and $P_h=P_{h,b}$ are smooth. 
Moreover,  
\begin{equation}\label{eq:isointsys}
\begin{tikzcd}
\mh(\cu,G)_{|\Bbo} \ar[r] \ar[d, "\bm{\chi}"'] &  \mh(\cu,G_h)_{|\Bbo}^\Cc\ar[d, "\bm{\chi}_h"] \\
\Bbo \ar[r] & \Bbo
\end{tikzcd}
\end{equation}
is an isomorphim of smooth algebraic integrable systems over $\Bbo$.
Here and in the following, we identify $\Bb$ with its image $\iota(\Bb)$ in $\Bb_h$ 
under the isomorphism $\iota\colon \Bb\to \Bb_h^\Cc$ induced by $\tfr/W\cong (\tfr_h/W_h)^\Cc$ according to \eqref{BisBhC}.
\end{thm}

\begin{proof}
Proposition \ref{p:foldingcameral} implies
that $\bc^{-1}(b)$ and $\bc_h^{-1}(b)$, $b\in \Bbo$, are smooth
since
$\bc^{-1}(b)\cong P_b$ and $\bc_h^{-1}(b)\cong P_{h,b}$, see Section \ref{sec:genericHitch}.

To prove that  \eqref{eq:isointsys} is an isomorphism of 
algebraic integrable systems, we first show that the top horizontal arrow of \eqref{eq:isointsys} is biholomorphic. 
Since both $\mh(\cu,G)$ and $\mh(\cu,G_h)$ are smooth over $\Bbo$ with fibers 
$P_b$ and $P_{h,b}$ over $b\in \Bbo$ respectively, 
it is sufficient to prove that the morphism 
\begin{equation}\label{eq:inducediso}
P_b\hookrightarrow P_{h,b}^\Cc
\end{equation}
is an isomorphism. 
It is an inclusion by Proposition \ref{inclusion}. 
Since the horizontal arrows of \eqref{eq:isointsys} are holomorphic by Proposition \ref{inclusion}, 
it is sufficient to prove \eqref{eq:inducediso} on the level of cocharacter lattices. 
Both $G$ and $G_h$ are of adjoint type, so that
\begin{equation*}
\mathrm{cochar}(P_b)=H^1(\cu,(p_{*b}\bLambda)^W),\quad 
\mathrm{cochar}(P_{h,b})=H^1(\cu,({p_{b}^{h}}_\ast\bLambda_h)^{W_h}),
\end{equation*}
see \eqref{eq:cocharP}.
By Proposition \ref{p:isosheaves}, 
$H^1(\cu, {p_b^h}_\ast \bLambda_h)^{W_h})\cong H^1(\cu,(p_{b*}\bLambda_h)^W)$.
Hence it remains to prove that the morphism 
\begin{equation*}
H^1(\cu,(p_{b*}\bLambda)^W)\longrightarrow H^1(\cu,(p_{b*}\bLambda_h)^W)^\Cc
\end{equation*}
induced by the inclusion $\bLambda\hookrightarrow \bLambda_h$ is an isomorphism. 
But this follows from \cite[Proposition 5.1.2]{Beck}. 
Therefore the top horizontal arrow of \eqref{eq:isointsys} is a biholomorphism. 

To see that it is a symplectomorphism, first note that both spaces in \eqref{eq:isointsys} are algebraic integrable systems over $\Bbo$ which admit Lagrangian sections, 
e.g.~Hitchin sections, see \cite[Lemma 4.1]{DP}.
The holomorphic symplectic structures therefore
induce symplectomorphisms
\begin{equation}\label{eq:isoTB}
\mh(\cu,G)_{|\Bbo}\cong T^*\Bbo/\Gamma,\quad \mh(\cu,G_h)_{|\Bbo}^\Cc\cong T^*\Bbo/\Gamma^\prime,
\end{equation}
see e.g. \cite[\S 3]{Freed}.
Here $\Gamma$ and $\Gamma^\prime$ are complex submanifolds of 
$T^*\Bbo$ which restrict to lattices in the fibers, and the quotients are taken fiberwise. 
They are Lagrangian with respect to the canonical symplectic structure 
$\omega_c$ on $T^*\Bbo$ so that $\omega_c$ descends to the quotients 
$T^*\Bbo/\Gamma$ and $T^*\Bb/\Gamma^\prime$. 
The isomorphisms \eqref{eq:isoTB} are compatible with \eqref{eq:isointsys} giving the commutative diagram
\begin{equation*}
\begin{tikzcd}
\mh(\cu,G)_{|\Bbo} \ar[r,"\cong"] \ar[d, "\cong"'] & \mh(\cu,G_h)^\Cc_{|\Bbo} \ar[d, "\cong"] \\
T^*\Bbo/\Gamma \ar[r] & T^*\Bbo/\Gamma^\prime,
\end{tikzcd}
\end{equation*}
where the vertical arrows are biholomorphic symplectomorphisms, while the 
upper horizontal arrow is biholomorphic. This implies that the lower horizontal
arrow is biholomorphic as well, and thus $\Gamma$ and $\Gamma^\prime$
are isomorphic. Since on both sides of this map, the symplectic structure is
induced by $\omega_c$, it follows that this map is a symplectomorphism.
Hence the upper horizontal arrow
 \eqref{eq:isointsys} is a symplectomorphism as well, concluding the proof. 
\end{proof}

\section{Crepant resolutions and Hitchin systems}\label{s:crepantresol}
Generalizing the results of \cite{DDP} for trivial $\Cc$, in \cite{Beck2},
a family of quasi-projective Gorenstein Calabi--Yau threefolds 
$\mathcal{X}\longrightarrow\Bb,\, \Bb=\Bb(\cu,G)$, with a $\Cc$-action 
was constructed for any cyclic group $\Cc$ of Dynkin graph automorphisms. 
Moreover, there it was shown that the canonical line bundles of these threefolds 
admit global nowhere-vanishing and $\Cc$-invariant sections. We say that 
their canonical classes are \emph{$\Cc$-trivializable}.
In \cite[Theorem 6]{Beck2}, the first author constructed an isomorphism
\begin{equation*}
J^2_\Cc(\X^\circ/\Bb^\circ)\cong \mh(\cu,G)_{|\Bbo},\quad 
\Bbo:=\Bbo(\cu,G),\quad 
\mathcal{X}^\circ:=\mathcal{X}_{|\Bbo},
\end{equation*}
of smooth algebraic integrable systems with section based on previous results from \cite{Beck}.
Here, for the fiber $X_b$ of $\X$ over $b\in\Bbo$,
\begin{equation*}
J^2_\Cc(X_b) :=H^3(X_b,\C)^\Cc/\left(F^2H^3(X_b,\C)^\Cc+H^3(X_b,\Z)^\Cc\right)
\end{equation*}
is isomorphic to the orbifold intermediate Jacobian
of the Calabi--Yau orbifold stack $[X_b/\Cc]$, see \cite[\S 4]{Beck2} for more details.

In this section, we construct simultaneous crepant resolutions of the singular quotients $X_b/\Cc$, $b\in \Bbo$. 
Moreover, we show that the Griffiths intermediate Jacobians of these crepant resolutions contain algebraic integrable systems that are isogenous 
to the folded Hitchin systems of Theorem \ref{thm:hitchinfibers}. 

\subsection{Quasi-projective Calabi--Yau threefolds over Hitchin bases}\label{ss:qcy}
As before, let $\gfr=\gfr(\Delta)$ denote the simple complex Lie algebra 
determined by $\Delta=\Delta_{h, \Cc}$, and $\cu$ a compact Riemann surface of genus 
\fontdimen16\textfont2=4pt
$g_\cu\geq 2$. 
\fontdimen16\textfont2=2.5pt
We choose a Slodowy slice
$\sigma\colon S\longrightarrow \tfr/W$ and a theta characteristic $L\longrightarrow\cu$, i.e. $L^2=K_{\Sigma}$. 
Using the $\C^*$-action on $S$ given in \eqref{CstaronS}, 
we twist $S$ by $L$, considered as a $\C^*$-bundle, to obtain the family of surfaces
\begin{equation}\label{eq:Scal}
\mathcal{S}:=L\times_{\C^*} S,\qquad
\bsigma\colon\mathcal{S} \longrightarrow L\times_{\C^*}\tfr/W.
\end{equation}
Here $\C^*$ acts with twice the usual weights on $\tfr/W$, cf.~Remark \ref{rem:Cstar}, 
so that $L\times_{\C^*}\tfr/W\cong \Ub$. 
Except for the case $(\Delta_h,\,\Cc)=(\mathrm A_{2k},\,\{\id\})$, $k\in\mathbb N$, we could alternatively work with a $\C^\ast$-action
on $S$ which squares to \eqref{CstaronS} and then twist by $K_\cu$ instead of $L$ in \eqref{eq:Scal},
as also follows from  Remark \ref{rem:Cstar}.

Pulling back the family $\mathcal{S}\longrightarrow \Ub$ of surfaces via the evaluation 
map $ev\colon\cu\times \Bb \longrightarrow \Ub$ and projecting to the second factor yields a family 
\begin{equation*}
\begin{tikzcd}
\X:=ev^*\mathcal{S} \ar[rr, bend left, "\bm{\pi}"] \ar[r] & \cu \times \Bb \ar[r] & \Bb
\end{tikzcd}
\end{equation*}
of  threefolds. 
By construction, the $\Cc$-action on $S$ induces a $\Cc$-action on $\mathcal{S}$ 
and therefore a $\Cc$-action on $\X$ as well. 
Since $\Cc$ acts trivially on the base $\Bb$ (see \eqref{BisBhC}),
the fibers $X_b:=\bm{\pi}^{-1}(b)$, $b\in \Bb$, 
are preserved by the $\Cc$-action. 
In \cite[\S 3.3.]{Beck2}, it is proven that $\X\longrightarrow \Bb$ is an 
algebraic family of quasi-projective Gorenstein threefolds with $\Cc$-trivializable
canonical class. 
For trivial $\Cc$, this was first observed in \cite{Sz1,DDP}.

By construction, each $X_b$, $b\in \Bb$, admits a map $X_b\to \cu$ which is affine.
This allows us to realize each of the threefolds $X_b$, $b\in \Bb$, as a hypersurface in 
the total space of a sum of powers of $K_\cu$ which is affine over $\cu$.
We next present the explicit equations for our running example and for $\Delta=\Delta_h^\Cc=\mathrm{G}_2$ with 
$\Delta_h=\mathrm{D}_4$ and $\Cc=\Z/3\Z$ which will be useful for later computations.
\begin{ex}\label{ex:A3cys}
For $\Delta=\mathrm{C}_2$, we choose $b=(b_2,b_4)\in 
\Bb$ using $\Bb\cong H^0(\cu, K_\cu^{2})\oplus H^0(\cu, K_\cu^{4})$, cf. \eqref{eq:runBBh}.
By Example \ref{ex:A3-sing}, using $L^2=K_\cu$ and writing tensor 
products as products, we have
\begin{equation*}
X_b=\{ (\alpha_1, \alpha_2,\alpha_3) \in \mathrm{tot}(K_\cu\oplus K_\cu^{2 }\oplus K_\cu^{2})
 ~|~[\alpha^4_1-\alpha_2\alpha_3 + b_2 \alpha_1^2 + b_4](\sigma) 
 =0,\, \sigma\in\cu \}.
\end{equation*}
The $\Cc$-action is defined by 
$(\alpha_1, \alpha_2,\alpha_3)\mapsto (-\alpha_1, \alpha_3,\alpha_2)$. 
Hence the $\Cc$-fixed point locus is
\begin{equation*}
X^\Cc=\{\alpha\in \mathrm{tot}(K_\cu^2)~|~[\alpha^2-b_4](\sigma)=0,\, \sigma\in \cu  \}
\end{equation*}
which has genus $6g_\cu-5$.
By varying $b\in \Bb$, we arrive at the family $\mathcal{X}\to \Bb$ with fiberwise $\Cc$-action.
Note that this is the global description of our family of Calabi--Yau threefolds
mentioned in \cite[Proof of Proposition 2.3]{DDP}, which is worked out in detail 
in \cite{Beck-thesis}.
\end{ex}

\begin{ex}\label{ex:g2}
Let $\Delta=\Delta_{h,\Cc}=\mathrm{G}_2$ with $\Delta_h=\mathrm{D}_4$ and $\Cc=\Z/3\Z$.
By \cite[\S 6.2(3)]{Slo}, the $\mathrm{\Delta_h}$-singularity $Y$
is determined by the equation
\begin{equation*}
x^3+y^3+z^2=0
\end{equation*}
with $\C^*$-action $(\lambda , (x,y,z))\mapsto (\lambda^4x, \lambda^4 y, \lambda^6 z)$.
Note that $(x,y,z)$ have weights $(2\mathrm w_x, 2\mathrm w_y, 2\mathrm w_z)$
with coprime  $(\mathrm w_x, \mathrm w_y, \mathrm w_z)=(2,2,3)$, in accord with Remark
\ref{rem:Cstar}.
The Jacobian ring is $\C[x,y,z]/(x^2, y^2, z)$, a quasi-homogeneous basis of which is represented by $(xy,\, x,\, y,\, 1)$. 
Thus a semi-universal $\C^\ast$-deformation
is given by 
\begin{eqnarray*}
&\mathcal Y_h:= \left\{ ( x,y, z, \beta_2, \beta_4, \widetilde \beta_4, \beta_6)\in\C^3\times\C^4 ~|~
x^3+y^3+z^2 +\beta_2xy+ \beta_4x+ \widetilde\beta_4y+ \beta_6 =0\right\},\\
&\mathcal Y_h\longrightarrow\C^4,\qquad
( x,y, z, \beta_2, \beta_4, \widetilde \beta_4, \beta_6) \mapsto (\beta_2, \beta_4, \widetilde\beta_4, \beta_6).
\end{eqnarray*}
As $\C^\ast$-action  we have
$$
\lambda\cdot ( x,y, z, \beta_2, \beta_4, \widetilde\beta_4, \beta_6)
= ( \lambda^4 x,\lambda^4y, \lambda^6z, 
\lambda^4\beta_2, \lambda^8\beta_4, \lambda^8\widetilde\beta_4, \lambda^{12}\beta_6).
$$
The $\Cc$-action is induced by $(x,y,z)\mapsto (\mu x,\mu^2 y, z)$ 
for any primitive third root $\mu$ of unity, which extends to
$( x,y,z, \beta_2, \beta_4, \widetilde\beta_4, \beta_6) \mapsto 
(\mu x,\mu^2 y, z, \beta_2, \mu^2\beta_4, \mu\widetilde\beta_4, \beta_6) $,
yielding the following semi-universal $\C^\ast$-deformation of $(Y,\Cc)$:
\begin{eqnarray*}
\left\{ ( x,y, z, \beta_2, \beta_6)\in\C^3\times\C^2 ~|~
x^3+y^3+z^2 +\beta_2xy+ \beta_6 =0\right\}
&\longrightarrow&\C^2,\\
( x,y, z, \beta_2,  \beta_6) &\mapsto& (\beta_2,  \beta_6).
\end{eqnarray*}
Using \eqref{Bdecomposes}, one finds that the Hitchin base 
$\Bb$ for $\Delta$ is isomorphic to $H^0(\cu,K_\cu^{2})\oplus H^0(\cu, K_\cu^{6})$.
Again fixing such an isomorphism we write any $b\in \Bb$ as $(b_2,b_6)$ for 
$b_j\in H^0(\cu, K_\cu^{\otimes j})$, $j\in \{ 2,6\}$.
Arguing as in the previous example, we see that the corresponding threefold $X_b$ is given by
\begin{equation*}
X_b=\{(\alpha_1, \alpha_2,\alpha_3)\in \mathrm{tot}(K_\cu^{2}\oplus K_\cu^{2}\oplus K_\cu^3)
~|~[\alpha^3_1+\alpha_2^3+\alpha_3^2 + b_2\alpha_1\alpha_2  + b_6](\sigma)=0,\, 
\sigma\in\cu \}.
\end{equation*}
The $\Cc$-action is induced by $(\alpha_1, \alpha_2,\alpha_3)\mapsto (\mu \alpha_1, \mu^2 \alpha_2,\alpha_3)$ so that 
$$
X^\Cc\cong \{ \gamma\in \mathrm{tot}(K_\cu^3)~|~[\gamma^2 +  b_6](\sigma)=0,\, \sigma\in\cu \}.
$$
Note that $X^\Cc$ has genus $8g_\cu-7$.
Varying $b\in \Bb$ yields the family $\mathcal{X}\to \Bb$ with $\Cc$-action. 
\end{ex}
\begin{rem}
While Slodowy slices allow a uniform construction of the families $\bm{\pi}\colon \mathcal{X}\to \Bb$ 
and to prove some of their properties, e.g. that $\bm{\pi}$ is smooth over $\Bbo$ \cite[\S 3.3.]{Beck2}, 
it is often useful to work with explicit equations along the lines of the above examples.
\end{rem}

\subsection{Crepant resolutions and their intermediate Jacobians}\label{ss:crepantresolJ2}
In this section, we construct unique crepant resolutions of the quotients 
$X_b/\Cc$
by the action of $\Cc$ for the family of Calabi--Yau threefolds $\X$ 
introduced above, restricted to $\Bbo$.
To make sure that each quotient $X_b/\Cc$, $b\in \Bbo$, admits a crepant resolution, we need the following:
\begin{lem}\label{lem:finitecover}
Let $S=x+\ker \mathrm{ad}_y\subset \gfr$ be our choice 
of Slodowy slice with $\Cc$-action, and assume that $\Cc$ is non-trivial. 
Then the fixed point locus $S^\Cc\subset S$ is finite over $\tfr/W$. 
\end{lem}
\begin{proof}
Let $\bar{t}\in \tfr/W$ and 
$Y:=\sigma^{-1}(\bar{t})$ with smooth part $Y^{sm}\subset Y$. 
By Theorem \ref{thm:Slosli}, $Y^{sm}$ coincides with the set of regular elements of $\gfr$ in $Y$, i.e. $Y^{sm}=Y\cap \gfr^{reg}$. 
Moreover, the Kostant-Kirillov form on adjoint orbits, cf. \cite[Chapter 1]{ChrissGinzburg}, is known to
restrict to a symplectic form $\omega$ on $Y^{sm}$ by \cite[\S 7]{GG}.
As was explained in Section \ref{sec:Slodowy},
$\Cc$ is represented on the Slodowy slice by the adjoint action, hence
$\omega$ is left invariant under the $\Cc$-action, i.e. $a^\ast\omega=\omega$. 

Since $Y-Y^{sm}$ is finite by Theorem \ref{thm:Slosli} \ref{thm:slob}, it suffices to consider $Y^{sm}$. 
Therefore the rest of the proof is a standard application of 
Cartan's linearization trick \cite[Lemme 1]{Cartan},
which allows us to replace the automorphism $a\in\Cc$ by its linearization on the tangent space. 
Since $a^\ast\omega=\omega$, we obtain a special unitary map 
of finite order on $\C^2$, which is thus either the
identity or has the origin as its unique fixed point. Hence by the assumption that
$\Cc$ is non-trivial, the fixed points of $a$ in $Y$ are isolated. 
But $a$ is an algebraic automorphism so that $Y^a\subset Y$ is finite. 
\end{proof}

Lemma \ref{lem:finitecover} implies that each $X_b$, $b\in \Bbo$, has a one-dimensional fixed point locus $X_b^\Cc\subset X_b$ whose irreducible components 
descend to curves of singularities of type $\mathrm{A}_1$ or $\mathrm{A}_2$ in $X_b/\Cc$. 
Hence we may deduce

\begin{prop}\label{LemResol}
Let $X=X_b$, $b\in \Bbo$, be a smooth quasi-projective Calabi--Yau threefold with 
$\Cc$-action as constructed in Section \ref{ss:qcy}, and view
the fixed point locus $X^\Cc$ as a subset of $X/\Cc$. 
Then the blow up $\widetilde{X/\Cc}:=Bl_{X^{\Cc}}(X/\Cc)$ of $X/\Cc$ along $X^\Cc$
is the unique crepant resolution
\begin{equation*}
\widetilde{X/\Cc}\to X/\Cc
\end{equation*}
of $X/\Cc$.
\end{prop}

\begin{proof}
The following proof is mostly standard. We give some details which
will prove useful in our constructions below.

To see that $X/\Cc$ has a \emph{unique} crepant resolution, we may
analyse each fixed point $x\in X^\Cc$ by applying Cartan's linearization trick
analogously to the proof of Lemma \ref{lem:finitecover}. Here, by the results
of Lemma \ref{lem:finitecover}, a non-trivial automorphism in $\Cc$  
is linearized at any fixed point by a special unitary map on
$\C^3$ with  precisely one eigenvalue $1$. 
Hence every singular point of $X/\Cc$ possesses a local neighborhood of 
the form $\C\times (\C^2/\Cc)$ such that $0\in \C^2/\Cc$ is a singularity 
of type $\mathrm{A}_1$ or $\mathrm{A}_2$.
In these local neighborhoods, the blowup $Bl_{X^\Cc}(X/\Cc)$ is given by
	\begin{equation*}
	Bl_{\C\times \{ 0\}}(\C\times (\C^2/\Cc))\cong \C \times Bl_0(\C^2/\Cc).
	\end{equation*}
Since $\Cc=\Z/2\Z$ or $\Z/3\Z$, this blowup is smooth.
In fact, it is the unique crepant resolution of the local model $\C\times (\C^2/\Cc)$.
By uniqueness, the local crepant resolutions glue together to give the unique 
crepant resolution $\widetilde{X/\Cc}$ of $X/\Cc$, and by construction, it coincides with $Bl_{X^\Cc}(X/\Cc)$.
\end{proof}
The resolution described above is in fact a simultaneous
resolution for $\widetilde{\mathcal{X}^\circ/\Cc}\longrightarrow \Bbo$:

\begin{prop}\label{prop:simultaneous}
	The family $\mathcal{X}^\circ/\Cc\longrightarrow \Bbo$ admits a simultaneous crepant resolution 
	$\widetilde{\mathcal{X}^\circ/\Cc}\longrightarrow \Bbo$. 
\end{prop}
\begin{proof}
	Since $\mathcal{X}^\circ\longrightarrow \Bbo$ is smooth over $\Bbo$ and 
	$\Cc$ is finite, $\mathcal{X}^\circ$ is covered by affine $\Cc$-invariant open subsets 
	$D\cong D^\prime\times D^{\prime\prime}$ with $D^\prime\subset \Bbo$ and 
	$D^{\prime\prime}\subset X_b$ for all $b\in D^\prime$. 
	By Proposition \ref{LemResol}, we may assume without loss of generality that each
	$D/\Cc$ is either smooth or obeys
	\begin{equation}
	D/\Cc\cong D^\prime\times (D^{\prime\prime}/\Cc)\cong D^\prime\times \C\times (\C^2/\Cc),
	\end{equation}
	where $0\in(\C^2/\Cc)$ is a  singularity of type 
	$A_1$ or $A_2$. By the uniqueness of these resolutions, as in the proof of 
	Proposition \ref{LemResol},
	the resolutions $\widetilde{D/\Cc}\longrightarrow D/\Cc$ glue to
	\begin{equation*}
	\widetilde{\mathcal{X}^\circ/\Cc} \longrightarrow \mathcal{X}^\circ/\Cc
	\end{equation*}  
	over $\Bbo$. 
	The restriction to each $b\in \Bbo$ is the crepant resolution $\widetilde{X_b/\Cc}\longrightarrow X_b/\Cc$ 
	of Proposition \ref{LemResol} by construction. 
\end{proof}
We next determine the rational mixed Hodge structure ($\Q$-MHS) of the crepant resolution in Proposition \ref{LemResol}. 
As a preparation, we need to determine the exceptional divisor of the blowup. 
\begin{prop}\label{p:excdivisor}
For $b\in B^\circ$, let $X:=X_b$ and $Z:=\widetilde{X/\Cc}$. Moreover, let $E\hookrightarrow Z$ denote the exceptional divisor of the resolution $\rho\colon Z \to X/\Cc$. 
Then $\rho_{|E}\colon E\to X^\Cc$ is a $\mathbb{P}^1$-bundle if $\Cc=\Z/2\Z$. 
If $\Cc=\Z/3\Z$, then $E=E_1\cup E_2$ decomposes into two 
$\mathbb{P}^1$-bundles $\rho_{|E_j}\colon E_j\to X^\Cc$, $j\in \{ 1,2\}$.
\end{prop}
\begin{proof}
By Proposition \ref{LemResol}, $Z=Bl_{X^\Cc}(X)$. If $\Cc=\Z/2\Z$, 
it follows that $\rho_{|E}\colon E\to X^\Cc$ is a $\mathbb{P}^1$-bundle.
If $\Cc=\Z/3\Z$, then the fibers of $E\to X^\Cc$ consist of two 
transversally intersecting copies of $\mathbb{P}^1$. Using the fact that 
$X\to \cu$ is an affine morphism, we next show that $E$ decomposes 
globally into two $\mathbb{P}^1$-bundles over $\cu$.

We use the notations of Example \ref{ex:g2} for fixed 
$b=(b_2,b_6)\in \Bbo\subset H^0(\cu,K_\cu^2)\oplus H^0(\cu,K_\cu^6)$. 
Further we introduce the coordinates
\begin{equation}
\nu_1:=\alpha_1^3,\quad \nu_2:=\alpha_2^3, \quad \nu_3:=\alpha_1 \alpha_2, 
\quad \nu_4:=\alpha_3.
\end{equation}
on $X/\Cc$, such that with 
$M:=K_\cu^6\oplus K_\cu^6\oplus K_\cu^4\oplus K_\cu^3$,
\begin{equation*}
Z\cong\{ (\nu_1,\nu_2, \nu_3,\nu_4)\in \mathrm{tot}(M)~|~
\nu_1+\nu_2+\nu_4^2 + b_2\nu_3 +  b_6=0 ,\quad \nu_3^3-\nu_1\nu_2=0\}.
\end{equation*}	
Under this isomorphism, by what was said in Example \ref{ex:g2}, the singular locus $X^\Cc\subset X/\Cc$ is given by $\{ (0,0,0,\nu_4)\in X/\Cc~|~\nu_4^2 + b_6=0 \}$. 
To compute the exceptional divisor $E$ of $Z=Bl_{X^\Cc}(X/\Cc)$, note that $Bl_{X^\Cc}(X/\Cc)$ is the proper transform of $X/\Cc\subset \mathrm{tot}(M)$ 
in $Bl_{M^\prime_0}(M)$, where $M^\prime_0$ is the zero section of 
$M^\prime:=K_\cu^6\oplus K_{\cu}^6\oplus K_{\cu}^4$ viewed as a submanifold 
of $\mathrm{tot}(M)$. Hence we obtain $E$ as the proper transform over $X^\Cc\subset X/\Cc$.
Since $X/\Cc\to \cu$ is affine, the calculation essentially works as in the affine case and we obtain
\begin{equation}
E=\{ ((0,0,0,\nu_4),  
[\gamma_1:\gamma_2:\gamma_3])\in \mathrm{tot}(M)\times \mathbb{P}(M')~|~
\nu_4^2 + b_6=0, \quad \gamma_1\gamma_2=0 \}.
\end{equation}
In particular, $E=E_1\cup E_2$ and each $E_j$, $j\in \{1,2 \}$, 
is a $\mathbb{P}^1$-bundle over $X^\Cc$ as claimed.
\end{proof}

\begin{thm}\label{thm:MHS}
	Let $X=X_b,$ $b\in \Bbo$, and $p\colon X\to X/\Cc$ its quotient under the action of $\Cc=\langle a \rangle$. 
		Further let $Z:=\widetilde{X/\Cc}$ and $\rho\colon Z\to X/\Cc$ be the crepant resolution of $X/\Cc$.
	The $\Q$-MHS on $H^3(Z,\Q)$ is pure of weight $3$, and it is isomorphic to the direct sum
		\begin{equation}\label{eq:splitMHS}
		H^3(Z,\Q)\cong H^3(X,\Q)^\Cc\oplus H^1(X^\Cc,\Q)(-1)^{|a|-1}
		\end{equation}
		of pure $\Q$-Hodge structures of weight $3$.
		Here $H^{1}(X^\Cc,\Q)(-1)$ denotes the Tate twist of $H^{1}(X^\Cc)$ by the
			$\Q$-Hodge structure $\Q(-1)$ of weight $2$, shifting all weights by $2$.
\end{thm}
\begin{proof}
All cohomology in this proof is $\Q$-valued, so to save notation, we suppress all $\Q$-coefficients
in the following. By \cite[Corollary-Definition 5.37]{PS-MHS},
there is a long exact Mayer--Vietoris sequence of $\Q$-MHS, which reads
	\begin{equation}\label{eq:exseqMHS2}
	\begin{tikzcd}
	\cdots \ar[r] & H^2(E)  \ar[r,"\delta_2"] & H^3(X/\Cc) \ar[r, "(\rho^*{,}i^*)"] & H^3(Z)\oplus H^3(X^{\Cc}) \\
	&\ar[r, "i_E^*-\rho_{|E}^*"] & H^3(E) \ar[r, "\delta_3"] & H^{4}(X/\Cc) \ar[r] & \cdots
	\end{tikzcd}
	\end{equation}
	Here $i_E\colon E\hookrightarrow Z$ is the inclusion of the 
	exceptional divisor and $i\colon X^\Cc\to X/\Cc$ is the inclusion of the fixed curve.
	We now analyze the above sequence term-by-term.
	
	Firstly, $H^3(X^\Cc)=0$ because $\dim_{\R} X^\Cc=2$. 
	Secondly, since $\Cc$ is a finite group and $\Q$ is a field of characteristic $0$,
		\begin{equation}\label{eq:q}
		 p^\ast\colon H^k(X/\Cc)\to H^k(X)^\Cc 
		 \end{equation}
	 is an isomorphism of $\Q$-vector spaces (see 
	 e.g. \cite[\S 3, Corollary 2.3.]{BorelTransformationGroups}). 
		Furthermore $p^\ast$ is a morphism of $\Q$-MHS. 
		The category of $\Q$-MHS is abelian, so that
		$p^\ast$ is an isomorphism of $\Q$-MHS. 
		Moreover, since $X$ is smooth (though non-compact), the weights of
		$H^k(X)$ are bounded below by $k$, hence so are the weights of  $H^k(X/\Cc)\cong H^k(X)^\Cc$. 
		
   Next we show that the $\Q$-MHS on $H^k(E)$ is in fact pure of weight $k$.  
		More precisely, we claim that there is an isomorphism
		\begin{equation}\label{eq:HkE}
		H^k(E)\cong \left[ H^{k-2}(X^\Cc)(-1)\right]^{|a|-1} \qquad\mbox{ for } k\in\{2,\, 3\}
		\end{equation}
		of  $\Q$-MHS which is in fact an 
		isomorphism of pure Hodge structures since the right hand side is pure. 
		Indeed, if $\Cc=\Z/2\Z$, then
		by Proposition \ref{p:excdivisor} we know that
		$E\to X^\Cc$ is a $\mathbb{P}^1$-bundle over $X^\Cc$. 
		Hence the Leray--Hirsch theorem (see for example \cite[Theorem 7.33]{VoisinI}) 
		gives an isomorphism \eqref{eq:HkE} of pure $\Q$-Hodge structures 
		of weight $k$ with $|a|=2$. 
	If $\Cc=\Z/3\Z$, then by Proposition \ref{p:excdivisor} we know that
	$E=E_1\cup E_2$ in $Z$ decomposes into two $\mathbb{P}^1$-bundles over $X^\Cc$ with transverse intersection. 
	The Mayer--Vietoris sequence for the triple $(Z,E_1,E_2)$ yields the short exact sequence
\begin{equation}\label{eq:mv}
\begin{tikzcd}
0 \ar[r] & \mathrm{cok} \ar[r] & H^k(E) \ar[r] & \mathrm{ker} \ar[r] & 0
\end{tikzcd}
\end{equation}
	of $\Q$-MHS.
	Here $\mathrm{cok}$ is the cokernel of 
	$(i_1^*,i_2^*)\colon H^{k-1}(E_1)\oplus H^{k-1}(E_2)\to H^{k-1}(E_1\cap E_2)$ 
	where $i_j\colon E_j\hookrightarrow Z$, $j\in \{ 1,2\}$, is the inclusion. 
	Moreover, $\mathrm{ker}$ is the kernel of $i_1^*-i_2^*\colon H^k(E_1)\oplus H^k(E_2)\to H^k(E)$.
	Now \eqref{eq:mv} implies that $H^k(E)$ only has weights $k-1$ and $k$. 
	If $k\geq 2$, then $\mathrm{cok}=0$ because $(i_1^*,i_2^*)$ is then surjective. 
	Hence we again conclude by means of the Leray--Hirsch theorem that \eqref{eq:HkE} holds, now with $|a|=3$.
The previous observations imply that the morphisms 
$\delta_k\colon H^k(E)\to H^{k+1}(X/\Cc)$, $k\in \{2,3 \}$, have to vanish due to the weights of the $\Q$-MHS.
Combining this fact with the isomorphisms \eqref{eq:q} and \eqref{eq:HkE}, the long exact sequence \eqref{eq:exseqMHS2} yields the short exact sequence
		\begin{equation}\label{eq:exseqMHS}
	\begin{tikzcd}
	0 \ar[r] & H^3(X)^\Cc \ar[r] & H^3(Z) \ar[r] & H^1(X^\Cc)(-1)^{|a|-1} \ar[r] & 0
	\end{tikzcd}
	\end{equation}
of $\Q$-MHS.
The $\Q$-MHS on $H^3(X)^\Cc$ is known to be  pure of weight $3$ and polarizable 
(see \cite[Lemma 4.3.1]{Beck}, also \cite[Lemma 3.1]{DDP} in 
the case where $\Cc=\{\id\}$).
By choosing a polarization on each of the pure $\Q$-Hodge 
structures of weight $3$ in \eqref{eq:exseqMHS}, we arrive at an isomorphism \eqref{eq:splitMHS}.
\end{proof}
\begin{cor}\label{cor:JZisogeny}
	The intermediate Jacobian of the crepant resolution 
	$Z$ of $X/\Cc$, $X=X_b,$ $b\in \Bbo$, is an abelian variety 
	and admits the non-canonical isogeny decomposition
	\begin{equation}\label{eq:JZ}
	J^2(Z)\simeq J_{\Cc}^2(X)\times J(X^\Cc)^{|a|-1}. 
	\end{equation}
	Here $J(X^\Cc)$ is the Jacobian of the curve $X^\Cc\subset X$.
\end{cor}
\begin{proof}
	Since the Mayer--Vietoris sequence is defined over the integers, 
	the proof of Theorem \ref{thm:MHS} gives the short exact sequence
		\begin{equation*}
		\begin{tikzcd}
		0 \ar[r] & H^3(X/\Cc,\Z)\ar[r] & H^3(Z,\Z) \ar[r] & H^1(X^\Cc,\Z)(-1)^{|a|-1}\ar[r] & 0
		\end{tikzcd}
		\end{equation*}
		of polarizable $\Z$-Hodge structures of weight $3$.
		The $\Z$-Hodge structure $H^3(X,\Z)^\Cc$ is known to be 
		effective\footnote{A pure $\Z$-Hodge structure of non-negative weight $w$ is called effective if $H^{p,q}=\{0\}$ for all $(p,q)\notin \mathbb{N}^2$ with $p+q=w$.}
		and pure of weight $1$ up to a Tate twist by \cite[Lemma 3.1]{DDP} for $\Cc=\{ \id\}$ 
		and \cite[Corollary 4.3.2]{Beck} in general.
		Therefore both polarizable $\Z$-Hodge structures $H^3(X/\Cc,\Z)$ 
		and $H^3(Z,\Z)$ are effective of weight $1$ up to a Tate twist as well.
		In particular, the intermediate Jacobians $J^2(X/\Cc)$, 
		$J^2(Z)$ and $J^2_{\Cc}(X)$ are abelian varieties.
		Since $J^2(X/\Cc)$ is isogenous to $J^2_{\Cc}(X)$ by the previous 
		proof (see \eqref{eq:q}), Poincar\'{e} reduciblity (e.g. \cite[Theorem 6.20]{Debarre-tori}) 
		implies a non-canonical isogeny decomposition
		\begin{equation*}
		J^2(Z)\simeq J^2_\Cc(X)\times J(X^\Cc)^{|a|-1}.
		\end{equation*}
\end{proof}
\begin{rem}
	 In Examples \ref{ex:A3cys} and \ref{ex:g2} we have seen that the genus of $X_b^\Cc$, $b\in \Bbo$, is strictly larger than 
	 \fontdimen16\textfont2=4pt
	 $g_\cu\geq 2$. 
	 \fontdimen16\textfont2=2.5pt
	In particular, $X_b^\Cc$ is a non-trivial branched covering of $\cu$ and $J(X_b^\Cc)$ varies non-trivially in $b\in\Bbo$.
	By inspection of all semi-universal $\C^\ast$-deformations of $\mathrm{ADE}$-type singularities one immediately checks that this always holds when
	$\Cc$ is non-trivial.
\end{rem}
\begin{cor}\label{cor:JZ}
Let $\mathcal{Z}\to \mathcal{X}^\circ/\Cc$ be the simultaneous crepant resolution over $\Bbo$.
	Then the intermediate Jacobian fibration $J^2(\mathcal{Z}/\Bb^\circ)$ 
	admits a non-canonical fiberwise isogeny decomposition 
		\begin{equation}\label{eq:JZglobal}
		J^2(\mathcal{Z}/\Bb^\circ)\simeq J^2_{\Cc}(\mathcal{X}^\circ/\Bb^\circ)\times 
		J(\ \X^\Cc_{|\Bb^\circ})^{|a|-1}
		\end{equation}
		over $\Bbo$.
\end{cor}
\begin{proof}
The key observation is that Proposition \ref{p:excdivisor} works over $\Bbo$ 
by arguing similarly to the proof of Proposition \ref{prop:simultaneous}. 
Indeed, the proof of Theorem \ref{thm:MHS} globalizes to an isomorphism of 
polarizable rational variations of Hodge structures of weight $3$ over $\Bbo$. 
Then by the same reasoning as in the proof of Corollary \ref{cor:JZisogeny} 
we have a fiberwise isogeny \eqref{eq:JZglobal}.
\end{proof}
In \cite[Theorem 5.2.1]{Beck} it was proven that there is an isomorphism
\begin{equation*}
J^2_\Cc(\mathcal{X}^\circ/\Bb^\circ)\cong \mathcal{M}(\Sigma,G)_{|\Bbo}
\end{equation*}
of algebraic integrable systems over $\Bbo$.
These are in turn isomorphic to $\mathcal{M}(\Sigma,G_h)^\Cc_{|\Bbo}$ by 
Theorem \ref{thm:hitchinfibers}.
In particular, $J^2(\mathcal{Z}/\Bb^\circ)$ contains $\mathcal{M}(\Sigma,G)_{|\Bbo}$ 
up to isogeny by Corollary \ref{cor:JZ}.
This corollary also implies that $J^2(\mathcal{Z}/\Bb^\circ)\to \Bbo$ cannot admit the 
structure of an algebraic integrable system for dimensional reasons when $\Cc$ is non-trivial.
It is an intriguing question if $\mathcal{Z}\to \Bbo$ embeds into a larger family 
$\widehat{\mathcal{Z}}\to \widehat{\Bbo}$ such that 
$J^2(\widehat{\mathcal{Z}}/\Bb^\circ)\to \widehat{\Bbo}$ has the structure of an algebraic integrable system.
And if so, does passing to the crepant resolutions have an analogue for Hitchin systems?
We hope to return to these questions in future work.

\appendix

\section{An isomorphism $S\cong S_{h|(\tfr_h/W_h)^\Cc}$ in our running example}
\label{app:sh}
In Example \ref{runslices}, we constructed a Slodowy slice 
$S\subset \gfr(\Delta)$ with a $\Cc$-action for our running 
example $\Delta=\mathrm{C}_2=\mathrm{A}_{3,\Cc}$.
In the following, we construct a Slodowy slice 
$S_h\subset \gfr(\Delta_h)=\mathfrak{sl}_4(\C)$ with a $\Cc$-action 
and a $(\C^*\times \Cc)$-equivariant inclusion $S\hookrightarrow S_h$ 
which is an isomorphism over $(\tfr_h/W_h)^\Cc\cong \tfr/W$, in accordance 
with Theorem \ref{thm:Slosli}. We also construct an explicit 
$(\C^\ast\times \Cc)$-equivariant isomorphism of $S_h$ with the 
semi-universal $\C^\ast$-deformation obtained by means of the 
Jacobian ring for the surface singularity of type $\mathrm A_3$.

Consider the $\mathfrak{sl}_2$-triple
$(x_h,y_h,h_h)$ with
$$
x_h=\left( \begin{matrix} 0&1&1&0\\ 0&0&0&-1\\  0&0&0&-1\\  0&0&0&0\end{matrix}\right),\;\;
y_h=x_h^T=\left( \begin{matrix} 0&0&0&0\\ 1&0&0&0\\  1&0&0&0\\  0&-1&-1&0\end{matrix}\right),\;\;
h_h=[x_h,y_h] = \mbox{diag}\left(2,0,0,-2\right)
$$
in $\mathfrak{sl}_4(\C)$.
For the corresponding Slodowy slice we obtain
$$
S_h = x_h + \ker \mathrm{ad}_{y_h}
= \left\{ \left. \begin{pmatrix} u_1^-&1&1&0\\ u_1^+ - u_2^-&-u_1^-& 2u_1^-&-1\\ 
u_2^+ + u_2^-&2u_1^-&-u_1^-&-1\\u_3^-&u_2^- - u_2^+&-u_1^+ - u_2^-&u_1^-  \end{pmatrix}
\right| 
u_1^\pm,\,  u_2^\pm,\, u_3^-\in \C \right\}. 
$$
To lift the $\Cc$-action to $S_h$, according to Remark \ref{rem:CA} we need to pick
some automorphism in $\CA(x_h,y_h)$ as in \eqref{CAdef} which does not belong
to $\CA(x_h,y_h)^\circ$. One checks that setting
$$
\forall\, m\in\C^\ast\colon\;
M_m:= \left(\begin{matrix} m&0&0&0\\ 0&{1\over2}(m+m^{-3})&{1\over2}(m-m^{-3})&0\\[3pt]
0&{1\over2}(m-m^{-3})&{1\over2}(m+m^{-3})&0\\ 0&0&0&m \end{matrix}\right),\qquad
\varphi_a\colon A\longmapsto-A^{\widetilde T},
$$
we obtain
$$
\CA(x_h,y_h) = \left\{ \mbox{Ad}_{M_m}, \mbox{Ad}_{M_m}\circ\varphi_a \mid m\in \C^\ast\right\}.
$$
The unique nontrivial automorphism in $\CA(x_h,y_h)$ which indudes
$\alpha_1^\vee\leftrightarrow\alpha_3^\vee$ on $\tfr_h$ is $\varphi_a$, which thus generates our choice of $\Cc$-action on 
$\mathfrak{g}_h=\mathfrak{sl}_4(\C)$.
Since $\varphi_a$ induces the desired automorphism $a$ on $\Delta_h$,  by what was said in 
Remark \ref{rem:CA}, on $S_h$ we have the correct induced action
$u_k^+\mapsto u_k^+,\, u_k^-\mapsto-u_k^-$, yielding
$$
S_h^\Cc 
= \left\{ \left. \left( \begin{matrix} 0&1&1&0\\ u^+_1&0&0&-1\\ u^+_2&0&0&-1\\
0&-u^+_2&-u^+_1&0  \end{matrix} \right) \right|
u^+_1,\, u^+_2\in \C\right\}.
$$
As explained in Example \ref{runchis}, with respect to the standard coordinates
$\xi_h$ on $\mathfrak t_h/ W_h$ we may calculate 
\begin{eqnarray*}
	&&\xi_h\circ\chi_h\colon
	\left( \begin{matrix} u_1^-&1&1&0\\ u_1^+ - u_2^-&-u_1^-& 2u_1^-&-1\\
		u_2^+ + u_2^-&2u_1^-&-u_1^-&-1\\u_3^-&u_2^- - u_2^+&-u_1^+ - u_2^-&u_1^-  \end{matrix} \right)
	\qquad\qquad\qquad\qquad\qquad\qquad\\
	&&\hphantom{\xi_h\circ\chi_h\colon\xi_h\chi_h}
	\stackrel{\eqref{chisexterior}}{\longmapsto} \left( 
	\begin{matrix}-6(u_1^-)^2-2(u_1^+ + u_2^+)\\[3pt] -8(u_1^-)^3+4u_1^-(u_1^+ + u_2^+)-2u_3^-\\[3pt]
		-3(u_1^-)^4+6(u_1^-)^2(u_1^+ + u_2^+) + 6u_1^-u_3^- +(u_1^+ - u_2^+)^2 -4(u_{2}^-)^2\end{matrix}
	\right)^T
\end{eqnarray*}
which as one checks is $\Cc$-equivariant.

This form allows us to construct an explicit isomorphism between $S_h$ and
the semi-universal $\C^\ast$-deformation $\mathcal Y_h$ of the surface singularity
of type $\mathrm A_3$ given in Example~\ref{ex:A3-sing}. Indeed, setting 
$(b_2,b_3,b_4):=\xi_h\circ\chi_h(A)$ for $A\in S_h$, with parameters
$u_k^\pm$ as before, the above calculation yields
\begin{eqnarray*}
u_1^+ + u_2^+ &=& -{\textstyle{1\over2}} b_2-3(u_1^-)^2,\\[2pt]
u_3^- &=& -{\textstyle{1\over2}} b_3-4(u_1^-)^3+2u_1^-(u_1^+ + u_2^+) 
= -{\textstyle{1\over2}} b_3-10(u_1^-)^3-b_2u_1^-, \\[2pt]
b_4 &=&  -81(u_1^-)^4 -9(u_1^-)^2b_2 -3u_1^-  b_3
+(u_1^+ - u_2^+)^2 -4(u_{2}^-)^2.
\end{eqnarray*}
Hence
\begin{equation}\label{unfoldingsagree}
(x,\,y,\,z):= (3u_1^-, \, u_1^+ - u_2^+ + 2u_2^-, \, u_1^+ - u_2^+ -2u_2^-)
\end{equation}
yields
$$
S_h = \left\{ (x,y,z, b_2,b_4,b_6)\in\C^6~|~  b_4 =  -x^4 - b_2x^2 -  b_3 x +yz\right\}
$$
in accord with \eqref{YhA3}. One immediately checks that \eqref{unfoldingsagree}
is $\Cc$-equivariant, since by the above, $a$ acts by $u_k^\pm\mapsto\pm u_k^\pm$,
in accord with \eqref{eq:suHaction}. 

Next we obtain 
$$
\xi_h\circ\chi_h\left(S_h^\Cc\right)
=\left\{  \left(-2(u_1^+ + u_2^+),\, 0,\, (u_1^+ - u_2^+)^2\right) \mid u_1^+,\, u_2^+\in \C\right\}
=\xi_h \left(  (\mathfrak t_h/ W_h)^{\Cc} \right)
$$
by Example \ref{runchis}, and hence
$$
s_h(u_1^-,\,u_2^-,\,u_1^+,\,u_2^+)
:=
\left( \begin{matrix} u_1^-&1&1&0\\ u_1^+ - u_2^-&-u_1^-& 2u_1^-&-1\\ 
u_2^+ + u_2^-&2u_1^-&-u_1^-&-1\\-4(u_1^-)^3+2u_1^-(u_1^+ + u_2^+)&u_2^- - u_2^+&-u_1^+ - u_2^-&u_1^-  \end{matrix} \right), \;\; u_k^\pm\in\C
$$
parametrizes $\chi_h^{-1}\left( (\tfr_h/W_h)^\Cc \right)$. 
Moreover, we have
\begin{align}\label{homslice}
&\xi_h\circ\chi_h\colon
\chi_h^{-1}\left( (\tfr_h/W_h)^\Cc \right) \quad\longrightarrow\quad 
\xi_h\left((\tfr_h/W_h)^\Cc\right)=\C^2
\nonumber
\\
&s_h(u_1^-,\,u_2^-,\,u_1^+,\,u_2^+) \longmapsto 
\left( 
\begin{matrix}-6(u_1^-)^2-2(u_1^+ + u_2^+)\\ 0\\
-27(u_1^-)^4+18(u_1^-)^2(u_1^+ + u_2^+)  +(u_1^+ - u_2^+)^2 -4(u_2^-)^2\end{matrix}
\right)^T.
\end{align}
The $\C^\ast$-action of $\lambda\in\C^\ast$, $\lambda=e^t$, on $S_h$ is obtained by conjugation with the matrix $\exp(-th_h)=\mbox{diag}(\lambda^{-2},1,1,\lambda^2)$, followed by multiplication by $\lambda^2$. 
For the $u_k^\pm$ we thereby read off:
$$
\lambda\cdot(u_1^-,u_2^-,u_3^-, u_1^+,u_2^+)= (\lambda^2 u_1^-,\lambda^4 u_2^-,\lambda^6u_3^-, \lambda^4 u_1^+, \lambda^4 u_2^+).
$$
One now immediately checks that \eqref{unfoldingsagree} is
also $\C^\ast$-equivariant, if in \eqref{eq:suC} one doubles all weights,
as explained in Remark \ref{rem:Cstar}.

According to Theorem \ref{thm:Slosli}, there is a $(\C^\ast\times\Cc)$-equivariant 
isomorphism over $(\tfr_h/W_h)^\Cc\cong\tfr/W$ from $S$ to $\chi_h^{-1}\left( (\tfr_h/W_h)^\Cc \right)$.
This is best visible by choosing new parametrizations for the bases 
$(\tfr_h/W_h)^\Cc$ and $\tfr/W$ of our semi-universal $\C^\ast$-deformations. 
We let
$$
\widetilde\xi_h\circ\chi_h\colon \chi_h^{-1} \left(  (\tfr_h/W_h)^\Cc \right) \longrightarrow\C^2,\qquad
\widetilde\xi_h\circ \xi_h^{-1}(b_2^h,\,0,\, b_4^h) = (-{\textstyle{1\over6}}b_2^h,\, -b_4^h ),
$$
so by \eqref{homslice} we have
$$
\widetilde\xi_h\circ\chi_h \left(s_h(u_1^-,\,u_2^-,\,u_1^+,\,u_2^+)\right)
=
\left( \begin{matrix}
(u_1^-)^2+{\textstyle{1\over3}}(u_1^+ + u_2^+)\\[5pt]
27(u_1^-)^4-18(u_1^-)^2(u_1^+ + u_2^+)  -(u_1^+ - u_2^+)^2 +4(u_2^-)^2
\end{matrix}\right)^T.
$$
In the notation of Example \ref{runslices}, we similarly let
$$
\widetilde\xi\circ\chi\colon  S \longrightarrow\C^2,\qquad
\widetilde\xi\circ \xi^{-1} (b_2,\, b_4) = ({\textstyle{1\over2}}b_2,\, 9(b_4-{\textstyle{1\over4}}b_2^2) ),
$$
so by \eqref{inhomslice} we have
$$
\widetilde\xi\circ\chi \left(  s(v_1^-,\,v_2^-,\,v_1^+,\,v_2^+) \right)
=
( (v_1^-)^2-v_2^+,\quad
36(v_1^-)^2v_2^+  -9\left((v_2^-)^2 +(v_1^+)^2\right)).
$$
One now checks that
\begin{eqnarray*}
	\Phi\colon S&\longrightarrow& \chi_h^{-1} \left(  (\tfr_h/W_h)^\Cc \right), \\
	s(v_1^-,\,v_2^-,\,v_1^+,\,v_2^+) &\mapsto &
	s_h\left( {\textstyle\sqrt{2\over3}}v_1^-,\,{\textstyle{3i\over2}}v_2^-,\,
	{\textstyle{1\over2}}(v_1^-)^2+{\textstyle{3\over2}}\left(v_1^+-v_2^+\right),\,
	{\textstyle{1\over2}}(v_1^-)^2-{\textstyle{3\over2}}\left(v_1^++v_2^+\right)\right)
\end{eqnarray*}
yields a $(\C^\ast\times\Cc)$-equivariant isomorphism such that the diagram
\begin{equation}\label{eq:PhiPsi}
\begin{tikzcd}
S\ar[r, "\Phi"] \ar[r] \ar[d] & \chi_h^{-1}((\tfr_h/W_h)^\Cc) \ar[d] \\
\tfr/W\ar[r, "\widetilde\xi_h^{-1}\circ\widetilde\xi"] & (\tfr_h/W_h)^\Cc
\end{tikzcd}
\end{equation}
commutes.

Some aspects of this example have already been worked out in \cite[Appendix B]{Beck2}.
Beware that the triple $(x_0,y_0,h_0)$ in $\mathfrak{sl}_4(\C)$ 
given there is not an $\mathfrak{sl}_2$-triple 
but only satisfies $[x_0,y_0]=h_0$.
However, the corresponding slice still gives the correct semi-universal 
$\C^*$-deformation of a $\mathrm{B}_2$-singularity.
Moreover, there, the Lie algebra $\mathfrak{so_5}(\C)$ 
of type $\Delta=\mathrm{B}_2$ is used rather than the (isomorphic) Lie algebra $\mathfrak{sp_4}(\C)$ of type $\Delta=\mathrm{C}_2$.

\bibliographystyle{alpha}
\bibliography{bibtex1}

\end{document}